\pdfoutput=1 
\documentclass{amsart} 

\usepackage{tikz}

\usepackage{amssymb}
\usepackage{amsmath}
\usepackage{url}
\usepackage{bm}
\usepackage{enumitem}
\usepackage{graphicx}

\usepackage{hyperref}
\hypersetup{
 colorlinks = true,
 linkcolor = blue,
 filecolor = blue,
 urlcolor = blue,
 citecolor = blue,
}
\hypersetup{unicode=true}

\newcommand{\erase}[1]{}

\theoremstyle{plain}
\newtheorem{theorem}{Theorem}[section]       
\newtheorem{lemma}[theorem]{Lemma}           
\newtheorem{proposition}[theorem]{Proposition} 
\newtheorem{corollary}[theorem]{Corollary}   

\theoremstyle{definition}
\newtheorem{definition}[theorem]{Definition} 
\newtheorem{example}[theorem]{Example}       

\theoremstyle{remark}
\newtheorem{remark}[theorem]{Remark}         

\numberwithin{equation}{section}
\numberwithin{table}{section}
\numberwithin{figure}{section}

\newcommand{\FF}{\mathord{\mathbb{F}}}

\newcommand{\PP}{\mathord{\mathbb{P}}}
\newcommand{\QQ}{\mathord{\mathbb{Q}}}
\newcommand{\RR}{\mathord{\mathbb{R}}}

\newcommand{\XX}{\mathord{\mathbb{X}}}
\newcommand{\YY}{\mathord{\mathbb{Y}}}
\newcommand{\ZZ}{\mathord{\mathbb{Z}}}

\newcommand{\CCC}{\mathord{\mathcal{C}}}

\newcommand{\FFF}{\mathord{\mathcal{F}}}

\newcommand{\III}{\mathord{\mathcal{I}}}

\newcommand{\LLL}{\mathord{\mathcal{L}}}
\newcommand{\MMM}{\mathord{\mathcal{M}}}

\newcommand{\PPP}{\mathord{\mathcal{P}}}

\newcommand{\WWW}{\mathord{\mathcal{W}}}

\newcommand{\SSSS}{\mathord{\mathfrak{S}}}    

\newcommand{\vect}[1]{\boldsymbol{#1}}

\newcommand{\maprightsb}[1]{\xrightarrow[\smash{#1}]{}}

\newcommand{\inj}{\hookrightarrow}
\providecommand{\twoheadrightarrow}{\mathrel{\rightarrow\!\!\!\!\rightarrow}}
\newcommand{\surj}{\twoheadrightarrow}

\newcommand{\isom}{\xrightarrow{\raise -3pt \hbox{\scriptsize $\sim$} }}

\newcommand{\set}[2]{\{\,{#1}\mid {#2} \,\}}           
\newcommand{\bigset}[2]{\left\{\; {#1} \; \left\vert \; {#2} \;  \right.\right \}}
\newcommand{\bigsetR}[2]{\left\{\; \left. {#1} \; \right\vert \; {#2} \;  \right \}}

\newcommand{\angs}[1]{\langle {#1}  \rangle}           

\newcommand{\tensor}{\otimes}
\newcommand{\inv}{^{-1}}                 
\newcommand{\dual}{^{\vee}}               

\newcommand{\sprime}{^{\prime}}

\newcommand{\sperp}{^{\perp}}

\newcommand{\semidirectproduct}{\rtimes}    

\DeclareMathOperator{\Aut}{Aut}
\DeclareMathOperator{\Image}{Im}

\DeclareMathOperator{\Stab}{Stab}
\DeclareMathOperator{\Ker}{Ker}
\DeclareMathOperator{\Coker}{Coker}

\DeclareMathOperator{\Hom}{Hom}
\DeclareMathOperator{\rank}{rank}
\DeclareMathOperator{\pr}{pr}
\DeclareMathOperator{\disc}{disc}

\newcommand{\OG}{\mathord{\mathrm{O}}}

\newcommand{\id}{\mathord{\mathrm{id}}}

\newcommand{\intf}[1]{\langle #1 \rangle}

\newcommand{\mystruth}[1]{\phantom{\llap{\vrule height #1}}}
\newcommand{\mystrutd}[1]{\phantom{\llap{\vrule depth #1}}}
\newcommand{\mystruthd}[2]{\phantom{\llap{\vrule  height #1 depth #2}}}

\newcommand{\dotz}{\cdot 0}
\newcommand{\dotinf}{\cdot \infty}
\newcommand{\Invols}{\III}
\newcommand{\Cen}{\mathrm{Cen}}
\newcommand{\Leech}{\Lambda}
\newcommand{\Golay}{\CCC_{24}}
\newcommand{\pig}{\pi_{g}}

\newcommand{\vep}{\varepsilon}
\newcommand{\PowOmega}{2^{\Omega}}
\newcommand{\weyl}{\vect{w}}

\newcommand{\IIlat}{\mathrm{II}}
\newcommand{\typeILat}{\mathrm{I}}
\newcommand{\typeIILat}{\mathrm{II}}
\newcommand{\Lts}{L_{26}}
\newcommand{\PPPts}{\PPP_{26}}
\newcommand{\FFFts}{{\FFF}_{26}}
\newcommand{\enrinvol}{\epsilon}

\newcommand{\Sigmatwlv}{\varSigma_{12}}
\newcommand{\Laminated}{\varLambda}
\newcommand{\BW}{\mathrm{BW}}

\newcommand{\Lamsxtn}{\Laminated_{16}}
\newcommand{\barLeech}{\overline{\Leech}}
\newcommand{\compword}[1]{\overline{#1}}

\newcommand{\intfLeech}[1]{\intf{#1}_{\Leech}}
\newcommand{\intfL}[1]{\intf{#1}_{\hskip -1pt L}}
\newcommand{\intfM}[1]{\intf{#1}_{\hskip -1pt M}}

\newcommand{\tilg}{\tilde{g}}

\newcommand{\LeechRoot}[1]{r(#1)}

\newcommand{\AOG}{\mathord{\mathrm{AO}}}

\newcommand{\ConCham}{\mathord{C}}
\newcommand{\ConChamz}{\ConCham(\weyl_0)}

\newcommand{\walls}{\WWW}
\newcommand{\tilwalls}{\widetilde{\WWW}}
\newcommand{\tilXi}{\widetilde{\Xi}}
\newcommand{\Xitau}{\Xi_{\tau}}
\newcommand{\tilXitau}{\tilXi_{\tau}}

\newcommand{\Lplus}{L_{+}}
\newcommand{\Lminus}{L_{-}}
\newcommand{\qplus}{q_{+}}
\newcommand{\qminus}{q_{-}}
\newcommand{\gplus}{g_{+}}
\newcommand{\gminus}{g_{-}}
\newcommand{\Leechplus}{\Leech_{+}}
\newcommand{\Leechminus}{\Leech_{-}}
\newcommand{\etaplus}{{\eta}_{+}}
\newcommand{\etaminus}{{\eta}_{-}}
\newcommand{\lambdaplus}{{\lambda}_{+}}
\newcommand{\lambdaminus}{{\lambda}_{-}}
\newcommand{\FFFplus}{\FFF_{+}}
\newcommand{\prplus}{\pr_{+}}

\newcommand{\PPPplus}{\PPP_{+}}
\newcommand{\vGam}{\varGamma}
\newcommand{\tilvGam}{\widetilde{\varGamma}}

\newcommand{\prLplus}{\pr^{L}_{+}}
\newcommand{\prLminus}{\pr^{L}_{-}}
\newcommand{\prLeechplus}{\pr^{\Leech}_{+}}
\newcommand{\prLeechminus}{\pr^{\Leech}_{-}}

\newcommand{\barSigma}{\overline{\Sigma}}

\newcommand{\tauI}{\tau_{120}}
\newcommand{\tauII}{\tau_{135}}
\begin{document}
\title[Involutions in the affine Conway group]
{Involutions in the affine Conway group}
\author{Ichiro Shimada}
\address{Department of Mathematics,
Graduate School of Science,
Hiroshima University,
1-3-1 Kagamiyama,
Higashi-Hiroshima,
739-8526 JAPAN}
\email{ichiro-shimada@hiroshima-u.ac.jp}
%
%
\dedicatory{In memory of Professor Tetsuji Shioda}
\begin{abstract}
The affine Conway group is the group of affine isometries of the Leech lattice.
This group is isomorphic to the automorphism group of 
a standard fundamental domain for the action of 
the Weyl group on the even                                                                                                                                                                                                                                                                                                      hyperbolic lattice $\typeIILat_{1, 25}$ of rank $26$.
 \par
In this paper, we show that the affine Conway group has exactly nine conjugacy classes of involutions.
We investigate their properties
and show that, 
among these nine classes, 
one class can be regarded as an analogue of the class of Enriques involutions of K3 surfaces.
 \par
Motivated by possible applications to K3 and Enriques surfaces, 
we investigate in detail the orthogonal groups of the hyperbolic lattices 
arising as the invariant sublattices of involutions in $\typeIILat_{1, 25}$.
This computation is carried out using the Borcherds method.
Unlike the examples considered previously, 
 the induced chambers possess infinitely many walls.
\end{abstract}
\maketitle
\section{Introduction}
\subsection{Main results}
The \emph{affine Conway group} $\dotinf$ is the group of 
all affine isometries of the Leech lattice $\Leech$,
and is equal to the semidirect product $\dotz \ltimes \Leech$
of $\dotz$ and $\Leech$,
where $\dotz:=\OG(\Leech)$ is the orthogonal group of $\Leech$
acting on $\Leech$ from the right, and 
$(g, \tau)\in \dotz \ltimes \Leech$ acts on $\Leech$ as
\[
\lambda\;\;\mapsto\;\; \lambda^g+\tau.
\]
In this paper,
we classify the involutions in $\dotinf$ and investigate their properties. 
\par
An interesting outcome of our study is that 
several distinguished lattices appear naturally. 
In particular, the unimodular hyperbolic lattices $\typeILat_{1, 9}$ and $\typeIILat_{1, 9}$ of rank $10$
arise as the invariant sublattices of involutions of the natural action of $\dotinf$ on 
the even hyperbolic lattice $\Lts:=\typeIILat_{1, 25}$ of rank $26$.                                                                                                                                                                                                                                                                                                    
\par
Another interesting aspect of this work lies in its connection with algebraic geometry. 
Via the Borcherds method~\cite{MR913200, MR1654763},
the action of $\dotinf$ on $\Lts$
has found applications in the study of automorphism groups of K3 surfaces, 
Enriques surfaces, and, more recently, irreducible holomorphic symplectic (IHS) manifolds
 (see, for example,~\cite{MR4931809}).
 Our results can be interpreted as geometric results 
concerning some \emph{virtual} K3 surface $\XX$ and \emph{virtual} Enriques surface $\YY$.
 See Section~\ref{subsec:alggeom}.
\par
Our first main result is as follows.
\begin{theorem}\label{thm:main1}
There exist exactly 
$9=1+1+3+4$
conjugacy classes of involutions in the affine Conway group $\dotinf$.
\end{theorem}
The natural homomorphism $\dotinf \to \dotz$ maps
the conjugacy class $[\gamma]$ of an involution $\gamma=(g, \tau)$
in $\dotinf$ to the conjugacy class $[g]$ of the involution $g$ in $\dotz$. 
Conway~\cite[Chapter 10,~Section 3.7]{theCSbook}
showed that there exist exactly four conjugacy classes
$[\vep_8], [\vep_{12}], [\vep_{16}], [\vep_{24}]$ of involutions in $\dotz$,
where $\vep_n$ is an involution such that the rank of 
the invariant sublattice 
$\Ker(\id-\vep_n)$ of the Leech lattice $\Leech$ is $24-n$.
The decomposition $9=1+1+3+4$ in Theorem~\ref{thm:main1} 
is intended to show the number of conjugacy classes of involutions in $\dotinf$
over each of the conjugacy classes of involutions in $\dotz$.
\par
A lattice of rank 
$n>1$
 is said to be \emph{hyperbolic} if its associated real quadratic space has signature $(1,n-1)$.
 (Since our applications are to algebraic surfaces, we use the convention $(1,n-1)$ instead of $(n-1,1)$.)
A \emph{positive cone} $\PPP$ of a hyperbolic lattice $L$
is one of the two connected components 
of $\set{x\in L\tensor \RR}{\intf{x, x} > 0}$.
\par
The group $\dotinf$ was considered by Conway~\cite{MR690711}
 in his study of the orthogonal group of an even unimodular hyperbolic lattice $\IIlat_{1,25}$.
Note that $\IIlat_{1,25}$ is isomorphic to 
 \[
 \Lts:= U\oplus \Leech(-1), 
\]
where $U$ is the hyperbolic plane with the basis $\weyl_0, \weyl_1$
satisfying 
\begin{equation}\label{eq:w0w1}
\intf{\weyl_0, \weyl_0}=\intf{\weyl_1, \weyl_1}=0, 
\quad \intf{\weyl_0, \weyl_1}=1, 
\end{equation}
and $\Leech(-1)$ is the \emph{negative-definite} Leech lattice.
Let $\PPPts$ be 
the positive cone of $\Lts$ containing $\weyl_0$ in its closure, 
let $\OG(\Lts, \PPPts)$ be the stabilizer subgroup of $\PPPts$ in $\OG(\Lts)$, 
and 
let $W(\Lts)$ denote the Weyl group of $\Lts$
(see~Section~\ref{subsec:notationlat} for the terminology).
Then $\dotinf$ is embedded into $\OG(\Lts, \PPPts)$
as the stabilizer of $\weyl_0$.
Conway~\cite{MR690711} showed that 
 this subgroup is
the automorphism group
of a standard fundamental domain of the action of $W(\Lts)$ on $\PPPts$.
In particular, we have 
\[
\OG(\Lts, \PPPts)= W(\Lts) \rtimes\dotinf.
\]
\par
To state our next main result, 
we use the following notation.
Let $M$ be a lattice.
For an involution $g\in \OG(M)$, we put
\begin{equation}\label{eq:Mgpm}
M(g, +):=\set{x\in M}{x^g=x},
\quad
M(g, -):=\set{x\in M}{x^g=-x}.
\end{equation}
For a non-zero integer $m$,
let $M(m)$ denote the lattice obtained from $M$ by multiplying its bilinear form by
$m$.
For a group $G$ and an element $g\in G$,
we denote by $\Cen(G, g)$ the centralizer of $g$ in $G$.
\begin{table}
 %
\[
\begin{array}{lc c }
\gamma=(g, \tau) & \Lts(\gamma, +) & |\Lts(\gamma, -)_4|
\mystrutd{5pt} \\
\hline 
\gamma_{8}=(\vep_8, 0) & U\oplus \Lamsxtn(-1) & 240
\mystruthd{10pt}{5pt} \\
\hline
\gamma_{12}=(\vep_{12}, 0) & U\oplus \Sigmatwlv(-2) & 264
\mystruthd{10pt}{5pt} \\
\hline
\mystruth{10pt}
\gamma_{16, 1}=(\vep_{16}, 0) &U\oplus E_8(-2) &4320\\
\gamma_{16, 120} =(\vep_{16},\tauI) & \typeILat_{1, 9}(2) &2016\\
\gamma_{16, 135} =(\vep_{16},\tauII) & \typeIILat_{1, 9}(2) &2272
\mystrutd{5pt}\\
\hline
\mystruth{10pt}
\gamma_{24, 0}=( \vep_{24}, 0)& U &196560\\
\gamma_{24, 4}=(\vep_{24},t_4)& V_{4} & 102352 \\
\gamma_{24, 6}=( \vep_{24}, t_6) & V_{6} &98256\\
\gamma_{24, 8}=( \vep_{24}, t_8) & V_{8} &98256
\end{array}
\]
\vskip 10pt
\caption{$\Lts(\gamma, +)$ and $\Lts(\gamma,-)$}\label{table:isoms}
%
%
\end{table}
Then Theorem~\ref{thm:main1} is refined as follows:
\begin{theorem}\label{thm:main2}
Suppose that $\gamma=(g, \tau)$ is an involution
in $\dotinf$,
and let $n$ be the rank of 
 the anti-invariant sublattice $\Leech(g, -)$ of the involution $g\in \dotz$ on $\Leech$.
Then,
regarding $\gamma$ as an involution in $\OG(\Lts, \PPPts)$,
we see that 
the sublattice $\Lts(\gamma, -)$ of $\Lts$ is of rank $n$ and is isomorphic to a sublattice of $\Leech (-1)$, 
whereas $\Lts(\gamma, +)$ is a hyperbolic lattice of rank $26-n$.
\par
The nine conjugacy classes in Theorem~\ref{thm:main1} are distinguished by 
the isomorphism classes of $\Lts(\gamma, +)$ and $\Lts(\gamma,-)$ as in Table~\ref{table:isoms}.
\end{theorem}
In Table~\ref{table:isoms}, we use the following notation.
\begin{itemize}
\item
We denote by $|\Lts(\gamma, -)_4|$
the size of the set of vectors of norm $-4$ in the negative-definite lattice $\Lts(\gamma, -)$.
Note that, as a sublattice of $\Leech(-1)$, 
the even lattice $\Lts(\gamma, -)$ has no vectors of norm $-2$.
\item 
$\Laminated_{n} $ is 
the laminated lattice of rank $n$. 
See~\cite[Chapter 6]{theCSbook} for the definition.
In particular, we have $\Leech=\Laminated_{24}$, 
and $\Lamsxtn$ is isomorphic to the Barnes--Wall lattice $\BW_{16}$.
\item $E_8$ is the root lattice of type $E_8$, that is,
a positive-definite even unimodular lattice of rank $8$,
which is unique up to isomorphism.
We have $E_8(2)\cong \Laminated_{8}$.
\item $\Sigmatwlv$ is 
a positive-definite \emph{odd} unimodular lattice of rank $12$
without vectors of norm $1$.
Note that $\Sigmatwlv$ is unique up to isomorphism.
See~\cite[Table 16.7]{theCSbook}.
\item $\typeILat_{1,9}$ is an odd unimodular hyperbolic lattice of rank $10$, 
which is unique up to isomorphism,
and is isomorphic to the orthogonal direct sum of
 a copy of $[1]$ and nine copies of $[-1]$.
 \item $\typeIILat_{1,9}$ is an even unimodular hyperbolic lattice of rank $10$,
 which is unique up to isomorphism,
 and is isomorphic to $U\oplus E_8(-1)$.
\item $V_{2m}$ is a lattice of rank $2$
with a Gram matrix 
\begin{equation}\label{eq:GramVm}
\left[\begin{array}{cc} 0&2 \\ 2 & -2m \end{array}\right].
\end{equation}
The lattices $V_{2m}$ and $V_{2m\sprime}$ are isomorphic 
if and only if $m\equiv m\sprime \bmod 2$.
\end{itemize}
\par
\medskip
Conway~\cite[Chapter 10]{theCSbook} also calculated the centralizers $\Cen(\dotz, g)$
of the involutions $g$ in $\dotz$.
We extend this result to the calculation of the centralizers $\Cen(\dotinf, \gamma)$
of the involutions $\gamma$ in $\dotinf$. See Proposition~\ref{prop:exactCen}.
\par
Using the computation of $\Cen(\dotinf, \gamma)$
and the Borcherds method~\cite{MR913200, MR1654763},
we can investigate 
the hyperbolic lattice $\Lts(\gamma, +)$.
In this Introduction, 
we present only the results concerning
 the lattices $\Lts(\gamma, +)=\typeILat_{1,9}$ (corresponding to the involution $\gamma_{16,120}$) 
and $\Lts(\gamma, +)=\typeIILat_{1,9}$ (corresponding to the involution $\gamma_{16,135}$).
Let $L$ be a hyperbolic lattice, 
and let $\PPP$ be a positive cone of $L$.
 Let $\FFF$ be a locally finite family of hyperplanes in $\PPP$. 
 An \emph{$\FFF$-chamber} is a closed polyhedral cone in $\PPP$ 
 bounded by members of $\FFF$ whose interior is disjoint from every member of $\FFF$.
\begin{theorem}\label{thm:main3}
Let $L$ be either $\typeILat_{1,9}$ or $\typeIILat_{1,9}$. 
Then there exist a subgroup $\vGam$ 
of $\OG(L,\PPP)$ and a locally finite family $\FFF$ of hyperplanes in $\PPP$
such that the following holds:
\begin{enumerate}[label={\rm(\roman*)}]
\item
The index of $\vGam$ in $\OG(L,\PPP)$ is 
\[
[\OG(L,\PPP): \vGam]=\begin{cases}
240 & \textrm{if $L=\typeILat_{1,9}$}, \\
71145 & \textrm{if $L=\typeIILat_{1,9}$}. \\
\end{cases}
\]
\item The family $\FFF$ is invariant under the action of $\vGam$, 
and $\vGam$ acts transitively on the set of $\FFF$-chambers.
\item Fixing an $\FFF$-chamber $D_0$,
we see that the stabilizer $\Stab_{\vGam}(D_0)$ of $D_0$ in $\vGam$
 acts transitively on the set of walls of $D_0$.
 \item 
 The group 
 $\Stab_{\vGam}(D_0)$ is an extension of a finite group by a free abelian group of rank $8$.
\end{enumerate}
\end{theorem}
Thus we obtained a $\vGam$-invariant tessellation of the cone $\PPP$ by polyhedral cones.
\begin{corollary}\label{cor:CorOfmain3}
Let $g\in\vGam$ be an element that maps $D_0$ to an $\FFF$-chamber adjacent to $D_0$. 
Then $\vGam$ is generated by $\Stab_{\vGam}(D_0)$ together with $g$.
\end{corollary}
As mentioned above, we use the Borcherds method for the proof. 
So far, 
 this method has been applied only in situations where each $\FFF$-chamber has finitely many walls. 
 The novelty of our computation is that 
 the $\FFF$-chambers have infinitely many walls. 
%
\subsection{Motivation from algebraic geometry}\label{subsec:alggeom}
The Borcherds method~\cite{MR913200, MR1654763}
is applied to the computation of automorphism groups of K3 surfaces and Enriques surfaces.
The first geometric application of this method was given 
by Kondo~\cite{MR1618132} to Jacobian Kummer surfaces.
See~\cite{MR3456710, MR4759604} for generalization of the method and computational details,
and~\cite{MR4498439, MR4389670} for applications to Enriques surfaces.
The Borcherds method is based on the result of Conway~\cite{MR690711} on 
the hyperbolic lattice $\Lts$. 
We can interpret Conway's result 
as a statement about the nef-and-big cone of a \emph{virtual} K3 surface $\XX$
 whose N\'eron--Severi lattice $S_{\XX}$ is isomorphic to $\Lts$.
 In fact,
in~\cite{MR4674939}, 
we regard the Leech lattice as the Mordell--Weil lattice~\cite{MR1081832} of an elliptic fibration of this \emph{non-existing} $\XX$, 
and 
obtained a new construction of the Leech lattice
by considering other elliptic fibrations on  $\XX$.
In this context, 
if we ignore  the question of what the period of $\XX$ should be, 
we can consider $\dotinf$ as 
the automorphism group of $\XX$,
and our result in this paper can be regarded as the classification of involutions of $\XX$.
Then 
the conjugacy class of the involution $\enrinvol:=\gamma_{16, 135}$ can be regarded 
 as the conjugacy class of Enriques involutions of $\XX$,
 because of the following result of Keum~\cite{MR1060704}.
 \begin{proposition}
 Let $X$ be a complex $K3$ surface
with the N\'eron--Severi lattice $S_X$.
 Then an involution $\iota$ of $X$ is an Enriques involution
 if and only if the invariant sublattice $S_X(\iota, +)$ of 
 the action of $\iota$ on $S_X$ 
 is isomorphic to $\typeIILat_{1,9} (2)$,
 and its orthogonal complement $S_X(\iota, -)$
contains no $(-2)$-vectors.
 \qed
 \end{proposition}
Let $\YY=\XX/\angs{\enrinvol}$ denote the \emph{virtual} Enriques surface
corresponding to the Enriques involution $\enrinvol$.
The inclusion 
\[
\Lts(\enrinvol, +)=S_{\YY}(2)\cong \typeIILat_{1,9} (2)\;\; \inj \;\;S_{\XX}=\Lts
\]
is the embedding of the N\'eron--Severi lattices
induced by the \'etale double covering $\XX\to \YY$.
This embedding 
has already appeared in~\cite{MR4389670} as one of the $17$ embeddings 
of $\typeIILat_{1,9} (2)$ into $\Lts$.
The embeddings in~\cite{MR4389670} other than this embedding 
have applications in the computation of the automorphism groups
of various Enriques surfaces in~\cite{MR4498439}.
We hope that this new embedding  $\Lts(\enrinvol, +)\inj \Lts$
has some new geometric and/or lattice theoretic applications.
\par
Since the seminal work of Nikulin~\cite{MR633160}, 
involutions of K3 surfaces have been extensively studied.
The problem of enumerating all Enriques involutions 
(up to conjugacy) of a given K3 surface was
studied by several authors.
For example, Ohashi~\cite{MR2552959},
building on the work of Kondo~\cite{MR1618132}, 
 classified all Enriques involutions on Jacobian Kummer surfaces. 
 See also~\cite{MR4288644}.
\par
See  Remark~\ref{rem:further}
for a further exploration of this heuristic analogy.
\subsection{Plan of the paper}\label{subsec:plan}
In Section~\ref{sec:prelim},
we fix notation and terminology concerning lattices, 
review the theory of discriminant forms
due to Nikulin~\cite{MR525944}, 
and explain some computational methods.
In Section~\ref{sec:proofofmain1},
we review Conway's results~\cite[Chapter 10]{theCSbook} on $\dotz$
and present the nine conjugacy classes of involutions in $\dotinf$.
Theorem~\ref{thm:main1} is proved in this section.
In Section~\ref{sec:proofThm2},
we investigate the action of involutions of $\dotinf$ on $\Lts$,
and prove Theorem~\ref{thm:main2}.
In Section~\ref{sec:Borcherds},
we apply the Borcherds method to the invariant sublattices of the involutions 
of $\dotinf$ in $\Lts$.
In Section~\ref{sec:wallsD0},
we present the detailed computational results obtained by
 the Borcherds method.
 Theorem~\ref{thm:main3} and similar results for other invariant lattices are proved in this section.
\par
We have used GAP~\cite{GAP4} 
for the proof of our main results.
Computational data are available from~\cite{affConwayCompdata}.
\section{Preliminaries}\label{sec:prelim}
In Section~\ref{subsec:notationlat}, 
we fix notation and terminology about lattices.
In Section~\ref{subsec:discfinvol},
we apply the theory of discriminant forms
to  involutions in the orthogonal group of a lattice.
In Section~\ref{subsec:computational},
we explain some computational methods 
for finite groups.
Since this section is preliminary and technical, the reader may skip it on a first reading and return to it later when necessary.
\subsection{Notation and terminology about lattices}\label{subsec:notationlat}
For a lattice $L$ and a subset $A$ of $L\tensor \RR$,
we put
\[
\OG(L, A):=\set{g\in \OG(L)}{A^g=A}.
\]
As was already noted,
we say that a lattice of rank $n>1$ is \emph{hyperbolic}
if the signature of the associated real quadratic space is $(1, n-1)$.
For a hyperbolic lattice $L$ with a positive cone $\PPP$,
the group
$\OG(L, \PPP)$ is a subgroup of the group of isometries of the associated hyperbolic space.
\par 
A vector $r\in L$ is called a \emph{$(-2)$-vector} if $\intf{r, r}=-2$.
It defines a reflection 
$s_r\in \OG(L)$ by $x^{s_r}:=x+\intf{x, r} r$.
For an even hyperbolic lattice $L$ with a positive cone $\PPP$,
 the \emph{Weyl group} $W(L)$
is the subgroup of $\OG(L, \PPP)$
generated by all reflections $s_r$ with respect to $(-2)$-vectors $r$.
For a vector $v \in L\tensor\QQ$ with $\intf{v, v}<0$,
we put 
\[
(v)\sperp:=\set{x\in \PPP}{\intf{x, v}=0}.
\]
If $r$ is a $(-2)$-vector, then $(r)^\perp$ is the mirror of $s_r$.
\begin{definition}
A \emph{standard fundamental domain} of the action of $W(L)$ on $\PPP$
is the closure in $ \PPP$ of a connected component of 
\[
\PPP\;\;\setminus \;\;\bigcup\; (r)\sperp, 
\]
where the union is taken over all $(-2)$-vectors $r$.
\end{definition}
We have $\OG(L, \PPP)=W(L)\rtimes \OG(L, N)$
for any standard fundamental domain $N$ of the action of $W(L)$ on $\PPP$.
\subsection{Discriminant forms and involutions}\label{subsec:discfinvol}
\newcommand{\dLat}{M}
\newcommand{\dLatplus}{M_+}
\newcommand{\dLatminus}{M_-}
Let $\dLat$ be an even lattice, and 
let $\dLat\dual:=\Hom(\dLat, \ZZ)$ 
 denote its \emph{dual lattice}.
Identifying $\dLat\dual$ as a subset of $\dLat\tensor \QQ$,
we let $\OG(\dLat)$ act on $\dLat\dual$ from the right
(\emph{not} contragrediently).
The intersection form $\dLat \times \dLat \to \ZZ$
extends to 
an intersection form $\dLat\dual \times \dLat\dual \to \QQ$.
We define 
\[
q_\dLat\;\colon\; \dLat\dual/\dLat \to \QQ/2\ZZ
\]
by $q_\dLat(v \bmod \dLat):=\intf{v, v} \bmod 2\ZZ$ for $v\in \dLat\dual$.
We call $\dLat\dual/\dLat$ the \emph{discriminant group} of $\dLat$,
and $q_\dLat$ the \emph{discriminant form} of $\dLat$.
We denote by $\OG(q_\dLat)$ the automorphism group of 
the finite quadratic form $q_\dLat$.
Then we have a natural homomorphism
\[
\eta_\dLat\;\colon\; \OG(\dLat) \to \OG(q_\dLat).
\]
In this paper,
we encounter many finite quadratic forms 
such that the underlying groups are 
direct sums of finite copies of $\ZZ/2\ZZ$. 
 We use the following notation:
\begin{equation}\label{eq:uvab}
\begin{array}{l}
u:=\left(\;(\ZZ/2\ZZ)^2, \;\left[\begin{array}{cc}0&1/2 \\1/2 &0\end{array}\right]\;\right),\quad 
v:=\left(\;(\ZZ/2\ZZ)^2, \;\left[\begin{array}{cc}1&1/2 \\1/2 &1\end{array}\right]\;\right),\\
\\
a:=\left(\ZZ/2\ZZ, \;\left[1/2\right]\right),\quad
b:=\left(\ZZ/2\ZZ, \;\left[3/2\right]\right).
\end{array}
\end{equation}
\begin{remark}\label{rem:isomqs}
We put $q_{[\mu,\nu, \alpha, \beta]}:=u^{\mu}\oplus v^{\nu} \oplus a^{\alpha}\oplus b^{\beta}$. 
Then we have the following isomorphisms of finite quadratic forms:
%
%
\[
\renewcommand{\arraystretch}{1.1}
\begin{array}{ll}
q_{[ 2, 0, 0, 0]}\cong q_{[ 0, 2, 0, 0]}, &
q_{[ 0, 0, 4, 0]}\cong q_{[ 0, 0, 0, 4]}, \\
q_{[ 1, 0, 1, 0]}\cong q_{[ 0, 0, 2, 1 ]}, &
q_{[ 1, 0, 0, 1]}\cong q_{[ 0, 0, 1, 2 ]},\\
q_{[ 0, 1, 1, 0 ]}\cong q_{[ 0, 0, 0, 3 ]}, &
q_{[ 0, 1, 0, 1 ]}\cong q_{[ 0, 0, 3, 0]}.
\end{array}
\]
\end{remark}
We apply 
Nikulin's theory of discriminant forms~\cite{MR525944}
to the study of involutions in $\OG(\dLat)$.
Let $\vep \in \OG(\dLat)$ be an involution, 
and let
$\dLatplus:=\dLat(\vep,+)$ and
$\dLatminus:=\dLat(\vep,-)$
be the sublattices defined by~\eqref{eq:Mgpm}.
Then we have an embedding
\begin{equation}\label{eq:iotaL}
\Cen(\OG(\dLat), \vep)\;\inj\; \OG(\dLatplus)\times \OG(\dLatminus)
\end{equation}
by the restrictions $g\mapsto (g|\dLatplus, g|\dLatminus)$.
Assume that $\dLat$ is unimodular.
Then the submodule
\[
\dLat/(\dLatplus\oplus \dLatminus) \;\subset\; (\dLatplus\dual/\dLatplus)\times (\dLatminus\dual/\dLatminus)
\]
is the graph of an anti-isomorphism
$q_{\dLatplus} \cong -q_{\dLatminus}$
of finite quadratic forms,
which 
induces an isomorphism
\[
j_\dLat\;\colon\; \OG(q_{\dLatplus}) \cong \OG(q_{\dLatminus}).
\]
Consequently, the embedding~\eqref{eq:iotaL} identifies
$\Cen(\OG(\dLat),\vep)$ with
\begin{equation*}\label{eq:CenL}
\set{(\gplus, \gminus)\in \OG(\dLatplus)\times \OG(\dLatminus)}{j_\dLat(\eta_{\dLatplus}(\gplus))=\eta_{\dLatminus}(\gminus)},
\end{equation*}
that is, we have the following fiber-product diagram:
\begin{equation}\label{eq:groupfibprodCenvep}
\renewcommand{\arraystretch}{1.4}
\begin{array}{ccc}
\Cen(\OG(\dLat), \vep) &\to & \OG(\dLatminus) \\
\downarrow & \square& \downarrow \\
\OG(\dLatplus) & \to & \OG(q_{\dLatplus}) \cong \OG(q_{\dLatminus}).
\end{array}
\end{equation}
\subsection{Computational method}\label{subsec:computational}
In some arguments, 
we need to determine, for a positive- or negative-definite lattice $M$,  the sizes of $\OG(M)$,
$\OG(q_M)$, and/or 
the image of the natural homomorphism
$\eta_M\colon \OG(M) \to \OG(q_M)$.
In this section, we explain the method used in these computations.
\par
Let $M$ be a positive-definite lattice of rank $m$.
A finite generating set 
and the size of $\OG(M)$ are obtained by the famous Plesken-Souvignier algorithm~\cite{MR1484483}.
A generating set and the size of the automorphism group $\OG(q)$ of a finite quadratic form 
$q$ can also be computed using a similar stabilizer-chain method.
\begin{remark}
A closed formula for the size of $\OG(q)$ is given in~\cite[Theorem 5.2]{MR4730253}.
\end{remark}
%
%
The size of the image of 
the natural homomorphism $\eta_M\colon \OG(M)\to \OG(q_M)$ is calculated as follows.
Let $A:=M\dual/M$ denote the discriminant group.
We first compute a generating set $\{g_1, \dots, g_N\}$ of $\OG(M)$,
together with their images $\eta_M(g_i)\in \OG(q_M)$ under $\eta_M$.
Each $\eta_M(g_i)$ induces a permutation of the finite set $A$.
The size of the image of $\eta_M$ is 
then equal to the size of the subgroup of 
the symmetric group $\SSSS(A)$ generated by these permutations,
which can be computed using the Schreier--Sims algorithm.
\begin{remark}\label{rem:GAPadvantage}
Since the implementation of the Schreier--Sims algorithm in {\tt GAP}~\cite{GAP4} is enhanced by randomized methods,
this computation by {\tt GAP} is generally very fast.
We take full advantage of this high performance of {\tt GAP} in our study.
\end{remark}
We can calculate  $\Ker \eta_M$
as follows.
We fix a basis 
$b_1, \dots, b_{m}$ of $M\dual$, 
and identify elements $g$ of $\OG(M\dual)=\OG(M)$
with tuples $[b_1\sprime, \dots, b_{m}\sprime]$ of vectors 
$b_i\sprime\in M\dual$
satisfying
\[
\intf{b_i, b_j}=\intf{b_i\sprime, b_j\sprime} \;\;(i, j=1, \dots, m)
\]
via $b_i^g=b_i\sprime$ for $i=1, \dots, m$.
Then 
the subgroup
$\Ker \eta_M$ of $\OG(M)$
corresponds to the set of 
all tuples $[b_1\sprime, \dots, b_{m}\sprime]$
such that $b_i-b_i\sprime\in M$ holds for $i=1, \dots, m$.
Enumerating such tuples $[b_1\sprime, \dots, b_{m}\sprime]$
by the stabilizer chain method as in the Plesken-Souvignier algorithm, 
we obtain the size of $\Ker \eta_M$
and a finite generating set of $\Ker \eta_M$.
%
%
%
\section{Proof of Theorem~\ref{thm:main1}}\label{sec:proofofmain1}
In Section~\ref{subsec:zinf}, 
we investigate the relationship between involutions of $\dotinf$ and those of $\dotz$.
We also describe the centralizer $\Cen(\dotinf, \gamma)$
for an involution $\gamma\in \dotinf$.
In Section~\ref{subsec:ConwayInvols}, 
we review Conway's result~\cite[Chapter 10]{theCSbook} on $\dotz$, 
and present the four involutions $\vep_n$ ($n=8,12,16,24$) 
 in $\dotz$.
In Sections~\ref{subsec:vep8}--\ref{subsec:vep24}, 
we study the conjugacy class $[\vep_n]$ in $\dotz$ for each $n$ 
and determine the corresponding conjugacy classes in $\dotinf$. This completes the proof of Theorem~\ref{thm:main1}.
\subsection{Involutions of $\dotz$ and of $\dotinf$}\label{subsec:zinf}
We use the following notation.
For a group $G$, let $\Invols(G)$ denote the set of conjugacy classes of involutions in $G$.
For an involution $g\in G$, let $[g]\in \Invols(G)$ denote the conjugacy class of $g$.
\par
The product in $\dotinf$ is given by 
\[
(g_1, \tau_1)(g_2, \tau_2)=(g_1g_2, \tau_1^{g_2}+\tau_2).
\]
Hence, for $(g, \tau)\in \dotinf$, we have 
\begin{equation}\label{eq:gtauisinvoliff}
\textrm{\;\;$(g, \tau)$ is an involution}
\;\Longleftrightarrow\;
\textrm{
$g\in \dotz$ is an involution and $\tau\in \Leech(g, -)$}.
\end{equation}
Moreover, we have 
\begin{equation}\label{eq:hsgthsinv}
(h, \sigma)(g, \tau)(h, \sigma)\inv
=(hgh\inv, (\tau+\sigma^g-\sigma)^{h\inv}).
\end{equation}
Hence 
two involutions $(g, \tau)$ and $(g\sprime, \tau\sprime)$ of $\dotinf$ are conjugate if and only if
\begin{equation}\label{eq:conjiff}
\exists\; h\in \dotz \;\;
\exists\; \sigma\in \Leech \;\;
\textrm{such that} \;\;
g\sprime=h g h\inv
\;\textrm{and}\; \tau\sp{\prime h}=\tau+\sigma^{g}-\sigma.
\end{equation}
Therefore the homomorphism $\dotinf \to \dotz$ given by $(g, \tau)\mapsto g$
induces a surjective mapping
\[
\Pi\;\colon\; \Invols(\dotinf)\surj \Invols(\dotz).
\]
Let $g\in \dotz$ be an involution.
The fiber of $\Pi$ over $[g]\in\Invols(\dotz)$ is described as follows.
We denote by 
\[
\pig\;\colon\; \Leech \to \Leech (g, -)
\]
the homomorphism defined by $\lambda\mapsto \lambda-\lambda^g$.
Note that $\pig$ is $\Cen(\dotz, g)$-equivariant,
and hence $\Cen(\dotz, g)$ acts on the cokernel of $\pig$.
For $\tau\in \Leech (g, -)$, 
we denote by $[\tau]$ the element $\tau+\Image(\pig)$ of $\Coker(\pig)$,
and by $\angs{\tau}\in \Coker(\pig)/\Cen(\dotz, g)$ the orbit of
$[\tau]$ under the action of $\Cen(\dotz, g)$ on $ \Coker(\pig)$.
\par
The following proposition has already been proved in the more general setting in \cite[Lemma 4.13]{MR4931809}. 
For the convenience of the reader, however, we provide a proof in our simpler setting.
\begin{proposition}\label{prop:bijf}
Suppose that $[(g_1, \tau_1)]\in \Pi\inv([g])$,
so that there exists an element $h_1\in \dotz$
such that $g_1h_1=h_1 g $ holds.
Note that $\tau_1\sp{h_1}\in \Leech(g, -)$.
Then the map
$[(g_1, \tau_1)] \mapsto \angs{\tau_1\sp{h_1}}$ 
is well defined, and 
yields a bijection 
\begin{equation}\label{eq:Piinvg}
f\;\;\colon\;\; \Pi\inv([g])\;\;\cong\;\; \Coker(\pig)/\Cen(\dotz, g).
\end{equation}
\end{proposition}
\begin{proof}
Suppose that $[(g_2, \tau_2)]=[(g_1, \tau_1)] $.
Then, by~\eqref{eq:conjiff}, we have $h_{21}\in \dotz$ and $\sigma_{21}\in \Leech$ 
such that $g_2 h_{21}=h_{21} g_1$
and $\tau_2\sp{h_{21}}=\tau_1 +\sigma_{21}\sp{g_1}-\sigma_{21}$.
Let $h_2\in \dotz $ be an element 
such that $g_2 h_2=h_2 g $.
It is enough to show 
$\angs{\tau_2^{ h_2}}=\angs{\tau_1^{h_1}}$ holds.
Note that $k:=h_2\inv h_{21} h_1$ belongs to $\Cen(\dotz, g)$.
Then we have 
\[
\tau_2^{ h_2}=\tau_2^{ h_{21} h_1 k\inv}
=(\tau_1 +\sigma_{21}\sp{g_1}-\sigma_{21})^{ h_1 k\inv}
=(\tau_1\sp{ h_1}+(\sigma_{21}\sp{h_1 g}-\sigma_{21}\sp{h_1}))^{k\inv}.
\]
Thus $f$ is well-defined.
\par
Since $f([(g, \tau)])=\angs{\tau}$ for any $\tau\in \Leech (g, -)$, 
the surjectivity of $f$ is obvious.
Note that, if $[(g\sprime, \tau\sprime)]\in \Pi\inv([g])$,
then there exists $\tau_0\in \Coker(\pig)$ such that 
$[(g\sprime, \tau\sprime)]=[(g, \tau_0)]$, 
and hence
$f([(g\sprime, \tau\sprime)])=\angs{\tau_0}$.
Therefore, to show the injectivity of $f$,
it is enough to show that, 
for $\tau, \tau\sprime\in \Coker(\pig)$,
the equality $\angs{\tau}=\angs{\tau\sprime}$ implies 
$[(g, \tau)]=[(g, \tau\sprime)]$.
If $\angs{\tau}=\angs{\tau\sprime}$ holds,
then we have $k\in \Cen(\dotz, g)$ and $\sigma\in \Leech$
such that $\tau\sprime=\tau^k+\sigma-\sigma^g$.
Then $(k, \sigma)(g, \tau\sprime)(k, \sigma)\inv=(g, \tau)$ holds.
\end{proof}
Let $(g, \tau)$ be an involution in $\dotinf$.
We put
\[
\Sigma_g([\tau]) :=
\textrm{ the stabilizer of $[\tau]\in \Coker(\pig)$ in $\Cen(\dotz, g)$}.
\]
Then we have the following:
\begin{proposition}\label{prop:exactCen}
We have an exact sequence 
\begin{equation}\label{eq:exactCen}
0 \;\longrightarrow\; \Leech(g, +)\;\longrightarrow\; \Cen(\dotinf, (g, \tau)) 
\;\longrightarrow\;\Sigma_g([\tau]) \;\longrightarrow\; 1.
\end{equation}
\end{proposition}
\begin{proof}
By~\eqref{eq:hsgthsinv}, it follows that 
\begin{equation}\label{eq:hsCen}
(h, \sigma)\in \Cen(\dotinf, (g, \tau))
\;\; \Longleftrightarrow\;\; 
\textrm{$h\in \Cen(\dotz, g)$ and $\tau-\tau^h=\pig(\sigma)$}.
\end{equation}
Therefore we have the projection 
$\Cen(\dotinf, (g, \tau))\to \Cen(\dotz, g)$
given by $(h, \sigma)\mapsto h$.
The kernel of this projection is equal to 
\[
\set{(\id, \lambda)}{\lambda\in \Leech(g, +)}\;\;\cong\;\; \Leech(g, +).
\]
On the other hand, by~\eqref{eq:hsCen}, 
an element $h$ of $\Cen(\dotz, g)$
belongs to the image of this projection if and only if 
there exists a vector $\sigma\in \Leech$ such that 
$\tau-\tau^h=\pig(\sigma)$, or equivalently, if and only if 
$h\in \Sigma_g([\tau]) $.
Thus we obtain the exact sequence~\eqref{eq:exactCen}.
\end{proof}
\subsection{Conway's result on involutions of $\dotz$}\label{subsec:ConwayInvols}
Conway investigated involutions of $\dotz$ and 
computed their centralizers in \cite[Chapter 10, Section 3.7]{theCSbook}.
In this section, we review 
Conway's result.
\par
Let $\Omega$ be the set
$\PP^1(\FF_{23})=\{\infty, 0,1,  \dots, 22\}$,
and let $\Golay\subset \PowOmega$ be the extended binary Golay code,
that is, the extended quadratic residue code over $\FF_{23}$.
We also use the Miracle Octad Generator (MOG)
to denote elements of $\PowOmega$.
The elements of $\Omega$ and the MOG positions are identified 
using the standard labeling~\cite[Figure~11.7]{theCSbook}, 
which we reproduce in Figure~\ref{fig:stdMOG}.
\begin{figure}
\begin{tikzpicture}[scale=.55, every node/.style={font=\small}]
\draw (0,0) rectangle (6,4);
\draw (2,0) -- (2,4);
\draw (4,0) -- (4,4);
\draw (0,2) -- (6,2);
\node at (0.5,3.5) {$0$};
\node at (0.5,2.5) {$19$};
\node at (0.5,1.5) {$15$};
\node at (0.5,0.5) {$5$};
\node at (1.5,3.5) {$\infty$};
\node at (1.5,2.5) {$3$};
\node at (1.5,1.5) {$6$};
\node at (1.5,0.5) {$9$};
\node at (2.5,3.5) {$1$};
\node at (2.5,2.5) {$20$};
\node at (2.5,1.5) {$14$};
\node at (2.5,0.5) {$21$};
\node at (3.5,3.5) {$11$};
\node at (3.5,2.5) {$4$};
\node at (3.5,1.5) {$16$};
\node at (3.5,0.5) {$13$};
\node at (4.5,3.5) {$2$};
\node at (4.5,2.5) {$10$};
\node at (4.5,1.5) {$17$};
\node at (4.5,0.5) {$7$};
\node at (5.5,3.5) {$22$};
\node at (5.5,2.5) {$18$};
\node at (5.5,1.5) {$8$};
\node at (5.5,0.5) {$12$};
\end{tikzpicture}
\vskip -.2cm
\caption{The standard MOG labeling}\label{fig:stdMOG}
\end{figure}
 For $n\ge 0$, 
we denote by $\Golay(n)$ the set of codewords $w\in \Golay$ of size $n$.
Note that $\Golay(n)\ne \emptyset$ if and only if $n\in \{0,8,12,16,24\}$.
For $S\in \PowOmega$, let $v_S\in \ZZ^{\Omega}$ denote the vector
\[
v_{S}(x):=
\begin{cases}
1 & \textrm{if $x\in S$}, \\
0 & \textrm{if $x\notin S$}.
\end{cases}
\]
When $S=\{i\}$ for $i\in \Omega$, we write $v_i$ instead of $v_{\{i\}}$.
Let $\Leech$ be the submodule of $\ZZ^{\Omega}$ generated by 
the vectors $2v_C$,
where $C$ runs through $\Golay(8)$, and 
the vector $v_{\Omega}-4v_{\infty}$.
%
%
We equip $\ZZ^{\Omega}$ with the inner product $\intf{\phantom{i}, \phantom{i}}$ 
given by
\[
\intf{x, y}:=\frac{1}{\,8\,} \sum_{i\in\Omega} x(i) y(i).
\]
Then $\Leech$ with $\intf{\phantom{i}, \phantom{i}}$ is the Leech lattice.
For a non-zero codeword $C\in \Golay$,
we define $\vep(C)\in \OG(\ZZ^{\Omega})$ by 
\[
v_i^{\vep(C)}:=
\begin{cases}
-v_i & \textrm{if $i\in C$}, \\
v_i& \textrm{if $i\notin C$}.
\end{cases}
\]
Then $\vep(C)$ 
preserves 
$\Leech\subset \ZZ^{\Omega}$, and hence gives 
an involution $\vep(C)\in \dotz$. 
Since the Mathieu group $M_{24}\subset \SSSS(\Omega)$ acts transitively
on $\Golay(n)$ for each $n$, 
and $M_{24}$ is naturally embedded into $\dotz$,
the conjugacy class in $\dotz$ of $\vep(C)$ for $C\in \Golay(n)$ 
depends only on $n$.
We choose a codeword $C_n\in \Golay(n)$ and put
\[
\vep_n:=\vep(C_n).
\]
Note that $\vep_{24}=-\id$.
Then Conway~\cite[Chapter 10, Section 3.7]{theCSbook} proved the following:
\begin{theorem}[Conway]
The set $\Invols(\dotz)$ of conjugacy classes of involutions in $\dotz$ 
consists of $4$ elements $[\vep_8]$, $[\vep_{12}]$, $[\vep_{16}]$, $[\vep_{24}]$.
\qed
\end{theorem}
For  $C\in \Golay$,
we put $\overline{C}:=\Omega\setminus C\in \Golay$, 
and let 
$\ZZ^{C}\subset \ZZ^{\Omega}$ denote 
the submodule generated by the vectors $v_i$ ($i\in C$).
Then we have
\begin{equation}\label{eq:Leechveppm}
\Leech(\vep(C), -)=\Leech\cap \ZZ^{C},
\quad
\Leech(\vep(C), +)=\Leech\cap \ZZ^{\overline{C}}.
\end{equation}
For each $\vep_n$,
we compute Gram matrices of $\Leech(\vep_n, \pm)$ 
with respect to certain bases
and the matrix representation of the homomorphism 
\[
\pi_n:=\pi_{\vep_n}\colon \Leech \to \Leech(\vep_n, -); \quad \lambda\mapsto \lambda-\lambda^{\vep_n}
\]
with respect to these bases.
By these computations,
 we obtain the following.
 See Table~\ref{table:minuspart} for a summary.
 Recall that $\Laminated_{n}$ is 
 the laminated lattice of rank $n$
(see~\cite[Figure 6.2]{theCSbook}).
 Note that $\Laminated_8$
is isomorphic to $E_8(2)$,
and $\Lamsxtn$ is 
isomorphic to the Barnes--Wall lattice $\BW_{16}$.
 Recall also that
$\Sigmatwlv$ is an \emph{odd} unimodular positive-definite lattice of rank $12$
that contains no vectors of norm $1$.
Note that $\Sigmatwlv$ is unique up to isomorphism
(see~\cite[Table 16.7]{theCSbook}).
\begin{table}
\[
\begin{array}{lllll}
n & \Leech(\vep_n, +)& \Leech(\vep_n, -) & \disc(\Leech(\vep_n, -)) & \Coker (\pi_n) \\
\hline
8 & \Lamsxtn & \Laminated_{8} & {(\ZZ/2\ZZ)}^{8\;\;} & 0 \mystruth{12pt}\\
12 & \Sigmatwlv(2) & \Sigmatwlv(2) & {(\ZZ/2\ZZ)}^{12} & 0 \\
16 & \Laminated_{8} & \Lamsxtn & {(\ZZ/2\ZZ)}^{8\;\;} &{(\ZZ/2\ZZ)}^{8\;\;} \\
24& 0 & \Leech &0 &{(\ZZ/2\ZZ)}^{24} \\
\end{array}
\]
{\small 
$\disc (L)$ means the discriminant group $L\dual/L$.\phantom{aaaa}}
\smallskip
\caption{$\Leech(\vep_n, +)$ and $\Leech(\vep_n, -)$}\label{table:minuspart}
\end{table}
\begin{itemize}
\item $n=8$. 
We choose as $C_{8}$ the codeword 
$\{ 0,\infty, 19, 3, 15, 6, 5, 9\} $, 
which corresponds to 
the positions with $*$ in Figure~\ref{figure:codewords} of MOG.
Then we obtain 
\[
\Leech(\vep_8, -)\cong \Laminated_{8},
\quad 
\Leech(\vep_8, +)\cong \Lamsxtn.
\]
We confirm  that 
 $\pi_{8}$ is surjective.
%
\item $n=12$. 
We choose as $C_{12}$ the codeword 
$\{ 0,\infty, 1,19, 3, 20,10, 18, 14, 5, 9, 13\}$.
See also Figure~\ref{figure:codewords}.
Then we obtain 
\[
\Leech(\vep_{12}, -)\cong \Sigmatwlv(2),
\quad 
\Leech(\vep_{12}, +)\cong \Sigmatwlv(2).
\]
The homomorphism $\pi_{12}$ is surjective.
%
\item $n=16$. 
We choose as $C_{16}$ the complement $\Omega\setminus C_8$ of $C_8$. 
Then we have $\vep_{16}=-\vep_{8}$, and hence 
\[
\Leech(\vep_{16}, -)=\Leech(\vep_{8}, +)\cong \Lamsxtn,
\quad 
\Leech(\vep_{16}, +)=\Leech(\vep_8, -)\cong \Laminated_{8}.
\]
The cokernel of $\pi_{16}$ is isomorphic to $(\ZZ/2\ZZ)^8$.
\item $n=24$. We have $\vep_{24}=-\id$ and hence 
$\Leech(\vep_{24}, -)=\Leech$.
The cokernel of $\pi_{24}$ is $\Leech/2\Leech \cong (\ZZ/2\ZZ)^{24}$.
\end{itemize}
\begin{figure}
\begin{tikzpicture}[scale=.4, every node/.style={font=\small}]
\draw (0,0) rectangle (6,4);
\draw (2,0) -- (2,4);
\draw (4,0) -- (4,4);
\draw (0,2) -- (6,2);
\node at (0.5,3.5) {$\ast$};
\node at (0.5,2.5) {$\ast$};
\node at (0.5,1.5) {$\ast$};
\node at (0.5,0.5) {$\ast$};
\node at (1.5,3.5) {$\ast$};
\node at (1.5,2.5) {$\ast$};
\node at (1.5,1.5) {$\ast$};
\node at (1.5,0.5) {$\ast$};
\node at (3, -1.) {$C_8$};
\end{tikzpicture}
\qquad
\begin{tikzpicture}[scale=.4, every node/.style={font=\small}]
\draw (0,0) rectangle (6,4);
\draw (2,0) -- (2,4);
\draw (4,0) -- (4,4);
\draw (0,2) -- (6,2);
\node at (0.5,3.5) {$\ast$};
\node at (0.5,2.5) {$\ast$};
\node at (0.5,0.5) {$\ast$};
\node at (1.5,3.5) {$\ast$};
\node at (1.5,2.5) {$\ast$};
\node at (1.5,0.5) {$\ast$};
\node at (2.5,3.5) {$\ast$};
\node at (2.5,2.5) {$\ast$};
\node at (2.5,1.5) {$\ast$};
\node at (3.5,0.5) {$\ast$};
\node at (4.5,2.5) {$\ast$};
\node at (5.5,2.5) {$\ast$};
\node at (3, -1.) {$C_{12}$};
\end{tikzpicture}
\qquad
\begin{tikzpicture}[scale=.4, every node/.style={font=\small}]
\draw (0,0) rectangle (6,4);
\draw (2,0) -- (2,4);
\draw (4,0) -- (4,4);
\draw (0,2) -- (6,2);
\node at (2.5,3.5) {$\ast$};
\node at (2.5,2.5) {$\ast$};
\node at (2.5,1.5) {$\ast$};
\node at (2.5,0.5) {$\ast$};
\node at (3.5,3.5) {$\ast$};
\node at (3.5,2.5) {$\ast$};
\node at (3.5,1.5) {$\ast$};
\node at (3.5,0.5) {$\ast$};
\node at (4.5,3.5) {$\ast$};
\node at (4.5,2.5) {$\ast$};
\node at (4.5,1.5) {$\ast$};
\node at (4.5,0.5) {$\ast$};
\node at (5.5,3.5) {$\ast$};
\node at (5.5,2.5) {$\ast$};
\node at (5.5,1.5) {$\ast$};
\node at (5.5,0.5) {$\ast$};
\node at (3, -1.) {$C_{16}$};
\end{tikzpicture}
\vskip -.3cm
\caption{Codewords}\label{figure:codewords}
\end{figure}
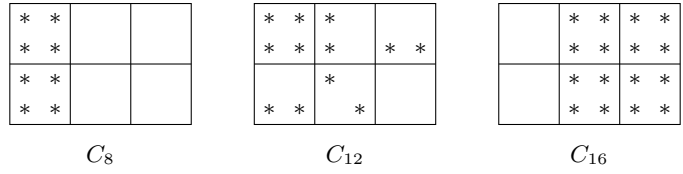
Our next task is to calculate $\Cen(\dotz, \vep_n)$
and its action on 
the cokernel of $\pi_n$.
Although these centralizers were already computed in \cite[Chapter~10]{theCSbook}, 
we rederive them using discriminant forms
 (see Section~\ref{subsec:discfinvol}).
%
%
\subsection{Involution $\vep_8$}\label{subsec:vep8}
We put
\[
G_8:=\OG^+_8(2), 
\]
which is a simple group of order 
$174182400$.
See~\cite[page 85]{MR827219}.
%
%
We put
\[
q_8:=q_{\Laminated_8}=q_{ E_8(2)},
\quad
q_{16}:=q_{\Lamsxtn}.
\]
Then $q_8$ 
is isomorphic to $u^{\oplus 4}$,
where $u$ is given in~\eqref{eq:uvab}.
Since $\Leech$ is unimodular, 
we have $q_{16}\cong -q_8$.
Since $u= -u$, 
we also have 
$q_{16}\cong u^{\oplus 4}$.
By~\cite[page~85]{MR827219}, 
we have
\[
\OG(q_8)\cong G_8.2, 
\quad 
\OG(\Laminated_8)=W(E_8)\cong 2.G_8.2.
\]
\begin{remark}\label{rem:2G2}
The homomorphism $\OG(q_8)\to \{\pm 1\}$ 
with the kernel $G_8$
is obtained by the permutation representation on
 the set of $4$-dimensional totally isotropic subspaces of 
$q_8$.
See~\cite[Section 3.8]{Wilson2009}.
The structure $W(E_8)\cong 2.G_8.2$ is given by
the inclusion $\{\pm \id\}\subset W(E_8)$ and $\det\colon W(E_8)\to \{\pm 1\}$.
\end{remark}
The group $\OG(\Laminated_8)=W(E_8)$
is generated by the $8$ reflections
corresponding to the standard fundamental roots of $E_8$.
Computing their actions on $q_8$,
we confirm that 
 the natural homomorphism $\OG(\Laminated_8)\to \OG(q_8)$ is surjective
with the kernel $\{\pm \id\}$.
We compute
the size $|\OG(\Lamsxtn)|$ and a finite set of generators of $ \OG(\Lamsxtn)$
by the method 
explained in Section~\ref{subsec:computational}.
%
%
%
%
 The size of $\OG(\Lamsxtn)$ is $512\cdot |G_8|$.
 Since $\Lamsxtn\cong \BW_{16}$, 
this result is consistent with $\OG(\BW_{16})\cong 2^{1+8}.G_8$
given in~\cite[Chapter 4, Section 10]{theCSbook}.
Using the method in Section~\ref{subsec:computational} again
and noting that $G_8$ is simple,
we see that 
\begin{equation}\label{eq:Imageeta16}
\parbox{11cm}{
\textrm{%
the image of 
the natural homomorphism $\eta_{16}\colon \OG(\Lamsxtn)\to \OG(q_{16})$  is \\
 of size $|G_8|$, and hence 
is equal to $G_8$ in $\OG(q_{16})\cong G_8.2$.
}
}
\end{equation}
%
%
We also  calculate
\[
K_{16}:=\Ker \,\eta_{16}
\]
by the method described  in Section~\ref{subsec:computational}.
It turns out that 
$K_{16}$ is 
an extraspecial $2$-group of plus type, namely $(2^{1+8})_{+}$.
%
%
%
%
%
%
%
\par 
Combining all these computations and using
the diagram~\eqref{eq:groupfibprodCenvep}, 
we obtain  
\[
\Cen(\dotz, \vep_8)\cong \{\pm \id\}\times \OG(\Lamsxtn).
\]
Since $\pi_{8}\colon \Leech \to \Leech(\vep_8, -)$ is surjective, 
Proposition~\ref{prop:bijf} implies that 
the subset $\Pi\inv ([\vep_8])$ of $ \Invols(\dotinf)$ consists of a single element,
 which is the conjugacy class 
 represented by an involution 
 \[
 \gamma_8:=(\vep_8, 0).
 \]
%
 %
%
%
%
%
\subsection{Involution $\vep_{12}$}\label{subsec:vep12}
We slightly simplify the argument of Conway~\cite[Chapter~10, Section~3.7]{theCSbook} 
by explicitly constructing the lattice $\Sigmatwlv$.
Let $\Sigmatwlv$ denote the submodule of $\ZZ^{12}$ generated by vectors 
$\pm 2 e_i\pm 2 e_j\; (i\ne j)$, 
and  $e_1+\cdots+e_{12}$,
where $e_1,\dots, e_{12}$ are the standard basis of $\ZZ^{12}$.
Then, by the intersection form 
\[
\intf{x, y}=\frac{1}{\,4\,} \sum_{i=1}^{12} x_i y_i
\]
on $\ZZ^{12}$, the submodule $\Sigmatwlv$ becomes the odd unimodular
positive-definite lattice with no vectors $v$ of norm $\intf{v, v}=1$.
%
\begin{lemma}\label{lem:frame}
The action of $\OG(\Sigmatwlv)$ on $\Sigmatwlv$ preserves
the \emph{frame}
\[
F:=\set{\pm 4 e_i}{i=1, \dots, 12}\subset \Sigmatwlv.
\]
\end{lemma}
\begin{proof}
Note that the set $(\Sigmatwlv)_2$ of vectors $u\in \Sigmatwlv$ with $\intf{u, u}=2$
consists of the $264$ vectors of the form $\pm 2 e_i\pm 2 e_j\; (i\ne j)$.
We say that $u, v, w\in( \Sigmatwlv)_2$ form a \emph{regular triple} if 
$\intf{u, v}=\intf{v, w}=\intf{w, u}=1$.
If $\{u, v, w\}$ is a regular triple, then the number of 
vectors $x\in( \Sigmatwlv)_2$ such that 
$\intf{u, x}=\intf{v, x}=\intf{w, x}=1$
is either $0$ or $16$.
We say that the regular triple $\{u, v, w\}$ is a \emph{triangle} if there exists no such $x\in (\Sigmatwlv)_2$.
See Figure~\ref{fig:triangle},
where $a, b, c , d\in \{2, -2\}$.
\begin{figure}
{\small
\[
\begin{array}{ccccccc}
u: & a& b & 0 & 0 & 0 &\dots\\
v: &a&0 & c & 0 & 0 &\dots\\
w: & 0& b & c & 0 & 0 & \dots\\
& \rlap{\textrm{\;\;\;\;\;triangle}} && &
\end{array}
\qquad\quad
\begin{array}{ccccccc}
u: &a& b & 0 & 0 & 0 &\dots\\
v: &a&0 & c & 0 & 0 & \dots\\
w: &a& 0 & 0 & d & 0 & \dots\\
\rlap{ \textrm{\;\;\;\;\;non-triangle}} && & &
\end{array}
\]
}
\vskip -1mm
\caption{Regular triples}\label{fig:triangle}
\end{figure}
There exist exactly $1760$ triangles, and 
the action of $\OG(\Sigmatwlv)$ preserves the set of triangles.
Since 
 the frame $F$ is equal to
\[
\set{u+v-w}{\textrm{$\{u, v, w\}$ is a triangle}},
\]
we conclude that $\OG(\Sigmatwlv)$ preserves $F$.
\end{proof}
\begin{corollary}\label{cor:OGSigma}
We have $\OG(\Sigmatwlv)=A\semidirectproduct\SSSS_{12}$,
where
\[
A:=\set{\alpha=(\alpha_1, \dots, \alpha_{12})\in \{\pm 1\}^{12}}{\prod\alpha_i=1}\;\; \cong\;\; (\ZZ/2\ZZ)^{11}
\]
acts on $\Sigmatwlv$ by $e_i\sp{\alpha}=\alpha_i e_i$ for $i=1, \dots, 12$, 
and
the symmetric group $\SSSS_{12}$ 
acts on $\Sigmatwlv$ by permuting the coordinates.
\qed
\end{corollary}
We fix a codeword $C_{12}\in \Golay(12)$ and 
consider the involution 
$\vep_{12}:=\vep(C_{12})$.
By~\eqref{eq:Leechveppm}, 
the lattice $\Leech (\vep_{12}, -)$ (resp.~$\Leech (\vep_{12}, +)$)
is isomorphic 
 to $\Sigmatwlv (2)$ under 
the isomorphism $\ZZ^{C_{12}}\cong \ZZ^{12}$ (resp.~$\ZZ^{\compword{C}_{12}}\cong \ZZ^{12}$)
induced by some (and hence any) bijection 
$C_{12}\cong \{1, \dots, 12\}$ (resp.~$\compword{C}_{12}\cong \{1, \dots, 12\}$),
and the frame $F\subset \Sigmatwlv$ corresponds to
\[
\set{\pm 8 v_i}{i\in C_{12}}\subset \Leech(\vep_{12}, -),
\qquad 
(\textrm{resp. }
\set{\pm 8 v_j}{j\in \compword{C}_{12}}\subset \Leech(\vep_{12}, +)),
\]
where $v_i$ ($i\in \Omega$) are the standard basis of $\ZZ^{\Omega}$.
Lemma~\ref{lem:frame}
implies that the action of $\Cen(\dotz, \vep_{12})$ 
on $\Leech$ preserves $\set{\pm 8 v_i}{i\in \Omega}$,
and hence, by Theorem~26 of~\cite[Chapter~10, Section~3.3]{theCSbook}, 
the group $\Cen(\dotz, \vep_{12}) $ is contained 
in the subgroup
\[
\angs{\vep(C) \mid C \in \Golay}\semidirectproduct M_{24} 
\]
of $\dotz$,
where $M_{24}\subset \SSSS(\Omega)$ is the automorphism group of $ \Golay$.
Note that, for any $C\in \Golay$,
the sizes of $C_{12}\cap C$ and $\compword{C}_{12}\cap C$
are even, and hence 
$\vep(C_{12}\cap C)$ and $\vep( \compword{C}_{12}\cap C)$ belong to $ A$ via 
$ \OG(\Leech(\vep_{12}, -))\cong \OG(\Sigmatwlv)$ and $ \OG(\Leech(\vep_{12}, +))\cong \OG(\Sigmatwlv)$,
respectively,
where $A$ is introduced in Corollary~\ref{cor:OGSigma}.
Thus we obtain 
\begin{equation}\label{eq:M12}
\Cen(\dotz, \vep_{12})=\angs{\vep(C) \mid C \in \Golay}\semidirectproduct \set{g\in M_{24}}{C_{12}^g=C_{12}}=(\ZZ/2\ZZ)^{12}\semidirectproduct  M_{12}.
\end{equation}
Since $\pi_{12}\colon \Leech\to \Leech(\vep_{12}, -)$ is surjective,
Proposition~\ref{prop:bijf} implies that the fiber $\Pi\inv ([\vep_{12}])\subset \Invols(\dotinf)$ consists of a single element,
 which is 
 represented by an involution 
 \[
 \gamma_{12}:=(\vep_{12}, 0).
 \]
 \begin{remark}\label{rem:q12}
 The discriminant form of $\Sigmatwlv(2)$ is isomorphic to
 $q_{12}:=a^{\oplus 12}$, 
and the size of $\OG(q_{12})$ is $51231497335603200$.
 By the method  explained  in Section~\ref{subsec:computational},
 we confirm that the kernel of the natural homomorphism from 
 $\OG(\Sigmatwlv(2))=\OG(\Sigmatwlv)$ to $\OG(q_{12})$ 
 is $\{\pm \id\}$.
 \end{remark}
\subsection{Involution $\vep_{16}$}\label{subsec:vep16}
Recall that we have $\Cen(\dotz, \vep_{16})=\Cen(\dotz, \vep_{8})$
by the choice of $C_{16}=\Omega\setminus C_8$.
We study the action of 
$\Cen(\dotz, \vep_{16})$
on the cokernel of the homomorphism 
$\pi_{16}\colon \Leech \to \Leech(\vep_{16}, -)=\Lamsxtn$.
Since $\Lamsxtn\dual/\Lamsxtn$ 
is $2$-elementary,
we have
\begin{equation} \label{eq:filteration}
\Lamsxtn\dual \;\supset\; \Lamsxtn\; \supset \; 2\Lamsxtn\dual.
\end{equation}
By direct computation,  we see  that 
the image  of $\pi_{16}$ is equal to $2\Lamsxtn\dual$.
Since the discriminant form $q_{16}$ of $\Lamsxtn$ 
is isomorphic to $u^{\oplus 4}$ and takes values in $\ZZ/2\ZZ$,
we have $\intf{x, x}\in \ZZ$ for any $x\in \Lamsxtn\dual$.
Therefore the map
\[
v+\Image (\pi_{16}) \;\mapsto\; \frac{\,1\,}{2} \intf{v, v} \;\bmod 2\ZZ
\]
gives rise to a quadratic form 
\[
\tilde{q}_{16}\;\colon\; \Coker (\pi_{16})\;\to\; \QQ/2\ZZ.
 \]
A direct computation shows that $\tilde{q}_{16}$ is isomorphic to $u^{\oplus 4}$.
The action of $\Cen(\dotz, \vep_{16})$ on $\Coker (\pi_{16})$ 
factors through the projection
$\Cen(\dotz, \vep_{16})\to \OG(\Lamsxtn)$,
which is surjective
by the diagram~\eqref{eq:groupfibprodCenvep} and the surjectivity of 
$W(E_8)\to \OG(q_8)$. 
Moreover, this action 
preserves $\tilde{q}_{16}$.
Since we have calculated a generating set of $\OG(\Lamsxtn)$,
we can explicitly compute the orbit decomposition of $\Coker (\pi_{16})$
by $\OG(\Lamsxtn)$; it consists of the following three orbits 
\[
\{0\},\quad 
\set{v \in\Coker (\pi_{16}) }{\tilde{q}_{16}(v)=1},
\quad
\set{v \in\Coker (\pi_{16}) }{v\ne 0, \;\; \tilde{q}_{16}(v)=0}.
\]
Their sizes are $1$, $120$, and $135$, respectively.
Hence, by Proposition~\ref{prop:bijf}, 
 it follows that the fiber $\Pi\inv ([\vep_{16}])\subset \Invols(\dotinf)$ consists of three conjugacy classes,
 which are 
 represented by involutions
 \[
 \gamma_{16, 1}:=(\vep_{16}, 0),
 \quad
 \gamma_{16, 120}:=(\vep_{16},\tauI),
 \quad
 \gamma_{16, 135}:=(\vep_{16},\tauII),
 \]
where $\tauI$ (resp.~$ \tauII$) is a non-zero vector of $\Leech(\vep_{16}, -)$ 
not in the image of $\pi_{16}$ 
such that
$\intf{\tauI,\tauI}\equiv 2 \bmod 4$
 (resp.~$\intf{\tauII,\tauII}\equiv 0 \bmod 4$).
 We choose from $\Leech(\vep_{16}, -)=\Lamsxtn$ a vector of norm $6$ as $ \tauI$,
 and a vector of norm $4$ as $ \tauII$.
%
(There exist exactly $4320$ vectors of norm $4$ 
and $61440$ vectors of norm $6$ 
in $\Lamsxtn$. 
None of them belongs to $2\Lamsxtn\dual=\Image(\pi_{16})$. )
\begin{remark}\label{rem:K16}
By direct computation, 
we find that 
the kernel $K_{16}$ of  the natural homomorphism 
$\eta_{16}\colon\OG(\Lamsxtn)\to \OG(q_{16})$
is in fact \emph{equal} to the kernel of 
the natural homomorphism 
$\tilde{\eta}_{16}\colon \OG(\Lamsxtn)\to \OG(\tilde{q}_{16})$.
Suppose that $g\in K_{16}$.
Then 
the action of $g$ on
the filtration~\eqref{eq:filteration} of $\Lamsxtn\dual$
induces the identities on 
 $\Lamsxtn\dual/\Lamsxtn$ and on $\Coker (\pi_{16})=\Lamsxtn/2\Lamsxtn\dual$, 
and hence gives rise to a homomorphism $\Lamsxtn\dual/\Lamsxtn\to \Coker (\pi_{16})$
by 
\[ x \bmod \Lamsxtn \;\;\mapsto \;\; x^g- x \bmod \Image(\pi_{16})
\]
for $x\in \Lamsxtn\dual$.
Thus we obtain an embedding
\[
K_{16}/\{\pm \id\}\inj \Hom(\Lamsxtn\dual/\Lamsxtn, \Coker (\pi_{16})), 
\]
which gives an isomorphism $K_{16}\cong (2^{1+8})_{+}$.
\end{remark}
%
%
\subsection{Involution $\vep_{24}$}\label{subsec:vep24}
We have $\vep_{24}=-\id$ and hence $\Leech(\vep_{24}, -)=\Leech$.
The orbit decomposition of 
$\Coker(\pi_{24})=\Leech/2\Leech$ by $\Cen(\dotz, \vep_{24})=\dotz$ 
was determined by  
Conway~\cite[Chapter 10, Section 3.3]{theCSbook}.
It consists of $4$ orbits $\{0\}, \barLeech(2), \barLeech(3), \barLeech(4)$, where 
\[
 \barLeech(n):=\set{\lambda\bmod 2\Leech}{\intf{\lambda, \lambda}=2n}.
\]
Their sizes are $1$, $98280$, $8386560$, and $8292375$, respectively.
Hence, by Proposition~\ref{prop:bijf}, we see that 
$\Pi\inv ([\vep_{24}])\subset \Invols(\dotinf)$ consists of four conjugacy classes,
 which are 
 represented by involutions
 \[
\gamma_{24, 0}:=(-\id, 0),
\quad
\gamma_{24, 4}:=(-\id, t_2),
\quad
\gamma_{24, 6}:=(-\id, t_3),
\quad
\gamma_{24, 8}:=(-\id, t_4),
 \]
where $t_m\in \Leech$ is a vector such that
$\intf{ t_m, t_m}=2m$.
\par
\medskip
Combining the results of Sections~\ref{subsec:vep8}, ~\ref{subsec:vep12}, ~\ref{subsec:vep16},~\ref{subsec:vep24},
we obtain a proof of Theorem~\ref{thm:main1}.
Using Proposition~\ref{prop:exactCen}, we also obtain descriptions of $\Cen(\dotinf, \gamma)$
for involutions $\gamma$ in $\dotinf$.
\section{Action of $\dotinf$ on $\Lts$}\label{sec:proofThm2}
In Section~\ref{subsec:AOGM}, 
we present a general theory on the action of affine isometries of an even positive-definite lattice $M$
on the hyperbolic lattice $L=U\oplus M(-1)$.
Applying this theory,
we  prove~Theorem~\ref{thm:main2} in Section~\ref{subsec:proofmain2}.
\subsection{Action of affine isometries on a hyperbolic lattice}\label{subsec:AOGM}
Let $M$ be an even positive-definite lattice.
We denote by 
\[
\AOG(M):=\OG(M)\ltimes M
\]
 the group of affine isometries of $M$
acting on $M$ as $x^{(g, \tau)}=x^g+\tau$
for $g\in \OG(M)$ and $\tau\in M$.
We consider the even hyperbolic lattice
\[
L:=U\oplus M(-1),
\]
where $U$ is the hyperbolic plane with the basis $\weyl_0, \weyl_1$
satisfying~\eqref{eq:w0w1}.
To avoid confusion, we write $\intfM{\phantom{a,a}}$
for the intersection form of $M$ (\emph{not} of $M(-1)$),
and $\intfL{\phantom{a,a}}$
for the intersection form of $L$.
An element of $L$ is written as 
\[
(a, b, x):=a\weyl_0+b \weyl_1+ x,
\]
where $a, b\in \ZZ$ and $x\in M$.
Then we have
\[
\intfL{\; (a, b, x), \; (a\sprime, b\sprime, x\sprime)\;}\;\;=\;\; ab\sprime+a\sprime b-\intfM{x, x\sprime}.
\]
Let $\PPP_L$ be the positive cone containing 
$\weyl_0=(1,0, 0)\in L$ in its closure.
It is easy to verify the following:
\begin{proposition}\label{prop:embeddingAOG}
The group $\AOG(M)$ acts on the lattice $L=U\oplus M(-1)$
by
\begin{align*}
\weyl_0^{(g, \tau)} \;\;=\;\; & \weyl_0, \\
\weyl_1^{(g, \tau)} \;\;=\;\; & \frac{\intfM{\tau, \tau}}{2}\;\weyl_0+\weyl_1+(0,0,\tau), \\
(0,0,x)^{(g, \tau)} \;\;=\;\; &\intfM{\tau, x^g} \;\weyl_0 + (0,0, x^g).
\end{align*}
Hence we have an injection $\AOG(M)\inj \OG(L, \PPP_L)$.
The image of this embedding is equal to
the stabilizer $\OG(L, \weyl_0)$ of $\weyl_0$ in $\OG(L, \PPP_L)$.
\qed
\end{proposition}
\begin{remark}\label{rem:tau0}
When $\tau=0$, the action of $(g, 0)\in \AOG(M)$ on $L$ 
is equal to $\id_U\oplus g$.
\end{remark}
Suppose that $(g, \tau)\in \AOG(M)$ is an involution, 
that is, $g\ne \id_M$, $g^2=\id_M$, and $\tau\in M(g, -)$.
We consider the sublattice 
\[
M(g, -,\tau):=\set{x\in M(g, -)}{\intfM{x, \tau}\equiv 0 \bmod 2}
\]
of $M(g, -)$.
Then the following is easy to confirm:
\begin{proposition}\label{prop:involL}
For an involution $\gamma:=(g, \tau)\in \AOG(M)$, we have 
\[
L(\gamma, -)=\bigsetR{\left(\frac{\intfM{x, \tau}}{2}, 0, x\right)\;\;}{x \in M(g, -,\tau)},
\]
which is isomorphic to 
$M(g, -,\tau)(-1)$.
In particular, 
$L(\gamma, -)$ is negative-definite of rank $\le \rank M$, and 
hence its orthogonal complement 
$L(\gamma, +)$ is hyperbolic.
\qed
\end{proposition}
\begin{remark}
The inclusion $M(-1)\inj L$ induces an anti-isomorphism $-q_M\cong q_L$.
The action of $\AOG(M)\subset \OG(L, \PPP)$ on $q_L\cong -q_M$ 
is the composite of the natural homomorphisms $\AOG(M) \to \OG(M)\to \OG(q_M)$.
\end{remark}
\subsection{Proof of Theorem~\ref{thm:main2}}\label{subsec:proofmain2}
We apply the general results in the previous section to the case where $M=\Leech$ and $L=\Lts$,
and obtain an isomorphism 
\[
\dotinf\cong \OG(\Lts, \weyl_0).
\]
For the $9$ conjugacy classes 
$[\gamma]=[(g, \tau)]$ of involutions of $\dotinf$,
we explicitly calculate the sublattices $\Lts(\gamma, +)$ and $\Lts(\gamma, -)$, 
and prove Theorem~\ref{thm:main2}.
\begin{remark}
For involutions represented by $\gamma=(\vep, 0)$,
the computation is trivial by Remark~\ref{rem:tau0}.
For involutions $\gamma=(\vep_{16},\tauI)$ and $\gamma=(\vep_{16},\tauII)$,
we make use of the uniqueness of the unimodular lattices $\typeILat_{1, 9}$ and $\typeIILat_{1, 9}$,
respectively.
\end{remark}
%
%
%
\begin{remark}
The lattice $\Lts(\gamma,+)$ for $\gamma=\gamma_{24,2m}=(-\id, t_m)$
with $2m \in \{4,6,8\}$ has a basis $\{\weyl_0, 2\weyl_1+\tau\}$.
under which the Gram matrix is equal to~\eqref{eq:GramVm}.
The isomorphism classes of the lattices $\Lts(\gamma,-)$ for $\gamma=\gamma_{24,6}$ 
and $\gamma=\gamma_{24,8}$ 
are distinguished by their discriminant forms,
which are anti-isomorphic to $q_{V_6}\cong a\oplus b$ and $q_{V_8}\cong u$, respectively. 
 \end{remark}
%
%
%
\section{Borcherds method for the invariant sublattices of involutions}\label{sec:Borcherds}
The Borcherds method is a method for computing a generating set of 
a subgroup of  $\OG(S, \PPP_S)$, 
where $S$ is  a hyperbolic primitive sublattice of $\Lts=U\oplus \Leech(-1)$.
In Sections~\ref{subsec:ConwayChamber} and~\ref{subsec:FFF},
we fix notions and terminology for the Borcherds method.
In Section~\ref{subsec:sublatLts}, 
we fix an involution
$\gamma=(\vep, \tau)\in \dotinf$, and 
apply the Borcherds method to the primitive hyperbolic sublattice
$\Lplus:=\Lts(\gamma, +)$.
We define a subgroup $\vGam$ of $ \OG(\Lplus, \PPPplus)$
with finite index, and establish some facts
for the computation of a  generating set of
 $\vGam$, which will be carried out in the next section.
\subsection{Conway chamber}\label{subsec:ConwayChamber}
Let $\PPPts$ be the 
positive cone of $\Lts$ 
containing $\weyl_0\in U$ in its closure.
We call standard fundamental domains of the action of $W(\Lts)$ on $\PPPts$ 
\emph{Conway chambers}.
A \emph{Weyl vector} is a non-zero primitive vector $\weyl \in \Lts$
with $\intf{\weyl , \weyl }=0$ contained in 
the closure of $\PPPts$ such that
the negative-definite lattice $[\weyl ]\sperp/[\weyl ]$ is isomorphic to $\Leech(-1)$,
where $[\weyl ]\sperp:=\set{v\in \Lts}{\intf{\weyl , v}=0}$.
For example, the vector $\weyl_0\in\Lts$ is a Weyl vector.
\par
Let $\weyl $ be a Weyl vector.
A $(-2)$-vector $r$ of $\Lts$ is called a \emph{Leech root with respect to $\weyl $} if $\intf{r, \weyl }=1$.
The set of all Leech roots with respect to $\weyl $ is denoted by $\LLL(\weyl )$.
We then put
\[
\ConCham(\weyl ):=\set{x\in\PPPts}{\intf{x, r}\ge 0\;\;\textrm{for all}\;\; r\in \LLL(\weyl )}.
\]
Conway~\cite{MR690711} proved the following:
\begin{theorem}[Conway]
The mapping $\weyl \mapsto \ConCham(\weyl)$ is a bijection from the set of Weyl vectors of $\Lts$ to
the set of Conway chambers.
\qed
\end{theorem}
\begin{corollary}
The embedding $\dotinf\inj \OG(\Lts, \PPPts)$ given in Section~\ref{subsec:AOGM} 
gives rise to an isomorphism 
$\dotinf\cong \OG(\Lts, \ConCham(\weyl_0))$.
\qed
\end{corollary}
We have 
\begin{equation}\label{eq:Leechrootlambda}
\LLL(\weyl_0)=\set{\LeechRoot{\lambda}}{\lambda\in \Leech},
\;\;\textrm{where}\;\;\; \LeechRoot{\lambda}:=\left(\frac{\intfLeech{\lambda, \lambda}-2}{2}, \; 1,\; \lambda\right).
\end{equation}
\subsection{$\FFF$-chambers and $\vGam$-simpleness}\label{subsec:FFF}
Let $S$ be a hyperbolic lattice
and $\PPP_S$ a positive cone of $S$.
Let $\FFF$ be a locally finite family of hyperplanes in $\PPP_S$.
An \emph{$\FFF$-chamber} is the closure in $\PPP_S$ of a connected component 
of the complement
\[
\PPP_S\;\;\setminus \;\; \bigcup_{H\in \FFF}\, H
\]
of the union of all hyperplanes in $\FFF$.
Let $D$ be an $\FFF$-chamber.
A \emph{wall} of $D$ is a closed subset of $D$
of the form $D\cap H$, where $H$ is a member of $\FFF$
such that $D\cap H$ contains a non-empty open subset of $H$.
We denote by $\walls(D)$ the set of walls of $D$.
A vector $v\in S\tensor\QQ$ is a \emph{defining vector} of a wall $w\in \walls(D)$ 
if $w=D\cap (v)\sperp$ and $\intf{v, a}>0$ for some (and hence any) interior point $a$ of $D$.
For a wall $D\cap H$ of $D$,
there exists a unique $\FFF$-chamber $D\sprime$
such that $D\ne D\sprime$ and that $D\cap H=D\sprime\cap H$.
We call this $D\sprime$ the \emph{$\FFF$-chamber adjacent to $D$ across the wall $D\cap H$}.
\par
Let $\vGam$ be a subgroup
of the group 
\[
\OG(S, \FFF):=\set{g\in \OG(S, \PPP_S)}{\textrm{the action of $g$ on $\PPP_S$ preserves $\FFF$}}.
\]
Then $\vGam$ acts on the set of $\FFF$-chambers.
We say that the set of $\FFF$-chambers is \emph{$\vGam$-simple}
if $\vGam$ acts transitively on the set of $\FFF$-chambers.
We fix an $\FFF$-chamber $D_0$, and put
\[
\Stab_{\vGam}(D_0):=\set{g\in \vGam}{D_0^g=D_0}.
\]
Then $\Stab_{\vGam}(D_0)$ acts on the set $\walls(D_0)$ of walls of $D_0$.
Let $\walls(D_0)/\Stab_{\vGam}(D_0)$ denote the set of 
$\Stab_{\vGam}(D_0)$-orbits in $\walls(D_0)$.
\begin{example}
Consider the case where $S=\Lts$ and 
\[
\FFF=\FFFts:= \set{(r)\sperp}{r\in \Lts,\;  \intf{r, r}=-2}.
\]
The Conway chambers are precisely the $\FFFts$-chambers,
and for a Weyl vector $\weyl$,
the set $\LLL(\weyl)$ 
is 
a set of 
defining vectors of all walls of $\ConCham(\weyl)$.
By  definition of $\FFFts$, 
we have $\OG(\Lts, \FFFts)=\OG(\Lts, \PPPts)$,
and for $\vGam=\OG(\Lts, \PPPts)$,
the set of Conway chambers  is $\vGam$-simple.
Moreover, we have $\Stab_{\vGam}(\ConChamz)=\dotinf$,
and $\dotinf$ 
 permutes 
the walls of $\ConChamz$ by 
$\LeechRoot{\lambda}^{(g, \tau)}=\LeechRoot{\lambda^g+\tau}$,
where $\LeechRoot{\lambda}\in \LLL(\weyl_0)$ is 
defined in~\eqref{eq:Leechrootlambda}.
\end{example}
The Borcherds method~\cite{MR913200, MR1654763} is 
based on the following observation:
\begin{proposition}\label{prop:observation}
Suppose that the set of $\FFF$-chambers is $\vGam$-simple.
For each orbit $o$ of the action of $\Stab_{\vGam}(D_0)$ on $\walls(D_0)$,
we choose an element $g_o\in \vGam$ such that
$D_0^{g_o}$ is adjacent to $D_0$ across a wall
belonging to $o$.
Then the group $\vGam$ is generated by the union of 
$\Stab_{\vGam}(D_0)$ and $\set{g_o}{o\in \walls(D_0)/\Stab_{\vGam}(D_0)}$.
\end{proposition}
\begin{proof}
Let $\varphi$ be an arbitrary element of $\vGam$.
We consider the $\FFF$-chamber $D_0^{\varphi}$.
Since $\PPP_S$ is connected, 
we have a sequence $D_0, D_1, \dots,  D_m=D_0^{\varphi}$
of $\FFF$-chambers such that $D_{i-1}$ and $D_i$ are 
adjacent across a wall for $i=1, \dots, m$.
We prove by induction on $i$ that there exists 
a product $g(i)$ of elements of $\Stab_{\vGam}(D_0)$ and 
 elements $g_o$ for $o\in \walls(D_0)/\Stab_{\vGam}(D_0)$
such that $D_i=D_0^{g(i)}$.
Then we have $\varphi=h g(m) $ for some $h\in \Stab_{\vGam}(D_0)$.
\end{proof}
%
%
\begin{remark}\label{rem:geom}
We give  several remarks on  applications of the Borcherds method to K3 and Enriques surfaces.
\par
(1) 
In these applications, 
the lattice $S$ is the N\'eron--Severi lattice of a surface $X$,
and $\vGam$ is the image of $\Aut(X)$ in $\OG(S,\PPP_S)$.
In many cases, the inclusion
$\vGam \subset \OG(S,\FFF)$
is guaranteed by the \emph{period condition} for $X$.
See~\cite[Section~5.2]{MR4759604}.
\par
(2) 
In some applications, 
the assumption of $\vGam$-simpleness fails.
See~\cite[Section 5.1]{MR4759604} for an algorithm
that takes the place of
Proposition~\ref{prop:observation}
in these cases.
\par
(3) 
In all geometric applications studied so far,
the number of walls of an $\FFF$-chamber is finite.
In the situations treated in this paper, 
the $\FFF$-chambers have infinitely many walls,
as will be explained in the next section.
\end{remark}
\subsection{Invariant sublattice of an involution in $\Lts$}\label{subsec:sublatLts}
We fix an involution
\begin{equation*}\label{eq:gamma}
\gamma=(\vep, \tau)\in \dotinf
\end{equation*}
 acting on $\Lts$, and 
 apply the Borcherds method to the primitive hyperbolic sublattice
\[
\Lplus:=\Lts(\gamma, +)
\]
of $\Lts$ 
with the positive cone $\PPPplus:=\PPPts\cap(\Lplus\tensor \RR)$.
Let
\[
\prplus\;\;\colon\;\; \Lts\tensor \RR \to \Lplus\tensor \RR
\]
denote the orthogonal projection.
A hyperplane $(r)\sperp\in \FFFts$ of $\PPPts$
defined by a $(-2)$-vector $r\in \Lts$
intersects $\PPPplus$
if and only if
$\intf{\prplus(r), \prplus(r)}<0$,
and in this case,
we have $ \PPPplus\cap (r)\sperp=(\prplus(r))\sperp$.
We put
\[
\FFFplus:=\set{(\prplus(r))\sperp}{\textrm{$r\in \Lts$ is a $(-2)$-vector satisfying 
$\intf{\prplus(r), \prplus(r)}<0$}},
\]
which is a locally finite family of hyperplanes in $\PPPplus$.
Then a closed subset $D$ of $\PPPplus$ is 
an $\FFFplus$-chamber if and only if the following hold:
\begin{enumerate}[label=(\alph*)]
\item 
$D$ contains a non-empty open subset of $\PPPplus$, and 
\item there exists a Conway chamber $\ConCham(\weyl)$
such that $D=\PPPplus\cap \ConCham(\weyl)$.
\end{enumerate}
We consider the following groups: 
\begin{equation}\label{eq:vGamdef}
\renewcommand{\arraystretch}{1.2}
\begin{array}{ccl}
\tilvGam &:=& \set{\tilde{g}\in \OG(\Lts, \PPPts)}{\Lplus^{\tilde{g}}=\Lplus}=\Cen(\OG(\Lts, \PPPts), \gamma), \\
\vGam &:=& \textrm{the image in $\OG(\Lplus, \PPPplus)$ of $\tilvGam$ under the restriction 
$\tilde{g}\mapsto \tilde{g}|\Lplus$}.
\end{array}
\end{equation}
%
Since the action of $\tilvGam$ on $\PPPts$ preserves $\FFFts$, we have 
\[
\vGam\;\;\subset \;\; \OG(\Lplus, \FFFplus).
\]
We will prove that 
the subgroup $\vGam$ of $ \OG(\Lplus, \PPPplus)$ is of finite index,
and that we can apply
Proposition~\ref{prop:observation}
to the set of $\FFFplus$-chambers
to compute $\vGam$.
\par
We put 
\[
\Lminus:=\Lts(\gamma, -),
\]
and let $\qplus$ and $\qminus$ be the discriminant forms 
of $\Lplus$ and $\Lminus$, respectively,
with the natural homomorphisms $\eta_{\pm}\colon \OG(L_{\pm})\to \OG(q_{\pm})$.
Since  $\Lts\subset \Lplus\dual\times \Lminus\dual$  is unimodular, 
it  induces an isomorphism $\qplus\cong \qminus$ and hence gives rise to 
$\OG(\qplus)\cong \OG(\qminus)$.
Then we have 
a  fiber-product diagram
\begin{equation}\label{eq:groupfibprodP}
\renewcommand{\arraystretch}{1.4}
\begin{array}{ccc}
\tilvGam &\to & \OG(\Lminus) \\
\downarrow & \square& \downarrow\rlap{ ${}^{\etaminus}$} \\
\OG(\Lplus, \PPPplus) & \maprightsb{\etaplus} & \OG(\qplus)\cong \OG(\qminus), 
\end{array}
\end{equation}
which is just 
a hyperbolic version of the diagram~\eqref{eq:groupfibprodCenvep}.
Since $\OG(q_{\pm})$ is finite,  the index of $\vGam$ in $\OG(\Lplus, \PPPplus)$ is finite.
Moreover, we also have the following:
\begin{proposition}\label{prop:vGamindex}
Suppose that  
$\etaplus\colon \OG(\Lplus, \PPPplus) \to \OG(\qplus)$ is surjective.
Then the  index of $\vGam$ in $\OG(\Lplus, \PPPplus)$  
is equal to  the index  of the image 
of $\etaminus\colon \OG(\Lminus) \to \OG(\qminus)$ 
in $\OG(\qminus)$.
\qed
\end{proposition}
\begin{proposition}\label{prop:interior}
Every $\FFFplus$-chamber $D=\PPPplus \cap \ConCham(\weyl)$
contains an interior point of the Conway chamber $ \ConCham(\weyl)$.
In particular, for every $\FFFplus$-chamber $D$,
the Conway chamber $ \ConCham(\weyl\sprime) $
such that $D=\PPPplus \cap \ConCham(\weyl\sprime)$ is unique.
\end{proposition}
\begin{proof}
Recall from Theorem~\ref{thm:main2} that 
the orthogonal complement $\Lminus$
 of $\Lplus$ in $\Lts$
is isomorphic to a sublattice of $\Leech(-1)$,
and hence does not contain any $(-2)$-vector.
Therefore no hyperplane $(r)\sperp\in \FFFts$ of $\PPPts$ defined by a $(-2)$-vector $r$ 
contains the subspace $\PPPplus$.
\end{proof}
\begin{proposition}\label{prop:D0}
The closed subset 
\[
D_0:=\PPPplus\cap \ConChamz
\]
of $\PPPplus$ is an $\FFFplus$-chamber.
\end{proposition}
\begin{proof}
It is enough to show that $D_0$ contains a non-empty open subset of $\PPPplus$.
To avoid confusion, we use $\intfL{\phantom{a, a}}$ to denote the intersection form of $\Lts$
and $\intfLeech{\phantom{a, a}}$ to denote that of $\Leech$.
For $x\in \Leech\tensor\RR$, we write $x^2$ for $\intfLeech{x,x}$.
Let $t$ be a real number with $t>2$.
We put
\[
p:=t\weyl_0+\weyl_1+\weyl_1^{\gamma}=\left(\; t+\frac{\tau^2}{2}, \;\;2, \;\; \tau\;\right).
\]
%
Then we have $p^{\gamma}=p$, $\intfL{p, \weyl_0}>0$, and $\intfL{p, p}>0$,
and hence $p\in \PPPplus$.
We have 
\[
\intfL{\LeechRoot{\lambda}, p}=t+\frac{\lambda^2}{2}+\frac{(\tau-\lambda)^2}{2}-2\ge t-2
\]
for any $\lambda\in \Leech$.
Therefore the point $p$ is an interior point of $D_0$ relative to $\PPPplus$.
\end{proof}
\begin{proposition}\label{prop:CenStab}
The restriction homomorphism
$\tilg\mapsto \tilg|\Lplus$
gives rise to a surjective homomorphism 
$\Cen(\dotinf, \gamma)\surj \Stab_{\vGam}(D_0)$.
\end{proposition}
\begin{proof}
Recall that 
$\tilvGam=\Cen(\OG(\Lts, \PPPts), \gamma)$.
Suppose that $\tilg\in \Cen(\dotinf, \gamma)$
and we put $g=\tilg|\Lplus$.
Then we have $g\in \vGam$.
Since $\tilg$ stabilizes $\ConChamz$,
we have $D_0^g=D_0$.
Therefore $ g\in \Stab_{\vGam}(D_0)$ holds.
To show the surjectivity,
let $g$ be an element of $\Stab_{\vGam}(D_0)$.
Since $g\in \vGam$,
we have an element $\tilg\in \tilvGam$ such that $g=\tilg|\Lplus$.
Proposition~\ref{prop:interior} implies that $\tilg$ stabilizes $\ConChamz$,
and hence $\tilg$ belongs to $\Cen(\dotinf, \gamma)$.
\end{proof}
Combining Proposition~\ref{prop:CenStab} with the exact sequence~\eqref{eq:exactCen}, we see that 
$\Stab_{\vGam}(D_0)$ is finitely generated,
and that a generating set can be computed.
\begin{proposition}\label{prop:PhiSimple}
The set of $\FFFplus$-chambers in $\PPPplus$ is $\vGam$-simple.
\end{proposition}
\begin{proof}
Let $D=\PPPplus \cap \ConCham(\weyl)$ be an $\FFFplus$-chamber.
We will find an element $\varphi$ of $\vGam$ 
that maps $D$ to $D_0$.
There exists an element $g$ of $W(\Lts)$
such that $\ConCham(\weyl)^g=\ConCham(\weyl_0)$.
Since ${\Lplus}^g=\Lts(g\inv \gamma g, +)$, 
we have $D^g\subset \Lts(g\inv \gamma g, +)\tensor \RR$.
Since $D$ contains an interior point of $\ConCham(\weyl)$ by 
Proposition~\ref{prop:interior},
the space $\Lts(g\inv \gamma g, +)\tensor \RR$ contains an interior point of $\ConCham(\weyl_0)$.
Therefore $g\inv \gamma g$ stabilizes $\ConCham(\weyl_0)$,
and hence $g\inv \gamma g \in \dotinf$.
The lattices $\Lts(g\inv \gamma g, + )$ and $\Lts(g\inv \gamma g, -)={\Lminus}^g$
are isomorphic to $\Lplus$ and $\Lminus$, respectively.
Hence, by Theorem~\ref{thm:main2},
the involutions $g\inv \gamma g$ and $\gamma$ are conjugate in $\dotinf$.
Thus we have an element $h\in \dotinf$
such that $\gamma g h =gh\gamma $.
It follows that $gh$ belongs to $\tilvGam$ and hence preserves $\Lplus$ and $\PPPplus$.
We then put
\[
\varphi:=gh|\Lplus\in \vGam.
\]
Since $\ConCham(\weyl)^{gh}=\ConCham(\weyl_0)$,
we obtain $D^\varphi=D_0$.
\end{proof}
\section{Walls of the $\FFFplus$-chamber $D_0$}\label{sec:wallsD0}
In this section, 
we fix an involution $\gamma=(\vep, \tau)$ of $\dotinf$ as in Section~\ref{subsec:sublatLts},
 and  investigate the hyperbolic lattice $\Lplus=\Lts(\gamma, +)$
 more closely.
 Here $\vep$ is one of $\vep_8, \vep_{12}, \vep_{16}$.
(The case $\vep=\vep_{24}$ is not interesting as $\rank \Lplus=2$.)
In Section~\ref{subsec:finite},
we show that 
the action of $\Stab_{\vGam}(D_0)$ on the set $\walls(D_0)$ of 
walls of $D_0$ has only finitely many orbits.
In Sections~\ref{subsec:walls8},~\ref{subsec:walls12},~\ref{subsec:walls16}, 
we compute the index of $\vGam$ in
$\OG(\Lplus, \PPPplus)$,
and describe the walls of $D_0$.
\par
In this section,
we continue to use
$\intfL{\phantom{a, a}}$ and $\intfLeech{\phantom{a, a}}$ 
to denote the intersection form of $\Lts$
and of $\Leech$, respectively.
For $x\in \Leech\tensor\RR$, we write $x^2$ for $\intfLeech{x,x}$,
as in the proof of Proposition~\ref{prop:D0}.
\subsection{Finiteness of the number of orbits in $\walls(D_0)$}\label{subsec:finite}
The following theorem combined with the results proved in the previous section
enables us to compute a finite generating set of $\vGam$
by applying Proposition~\ref{prop:observation} to $D_0=\PPPplus\cap \ConChamz$.
In the proof,
we present a method to enumerate the orbits explicitly.
\begin{theorem}
The action of $\Stab_{\vGam}(D_0)$ on $\walls(D_0)$ has only finitely many orbits.
\end{theorem}
\begin{proof}
We put
\[
\Lplus:=\Lts(\gamma, +),
\quad
\Lminus:=\Lts(\gamma,-), 
\quad
\Leechplus:=\Leech(\vep, +),
\quad
\Leechminus:=\Leech(\vep,-), 
\]
and denote the orthogonal projections by
\[
\prLplus \colon\Lts\to \Lplus\dual, \quad 
\prLminus\colon \Lts\to \Lminus\dual, \qquad 
\prLeechplus\colon \Leech\to \Leechplus\dual, \quad 
\prLeechminus\colon\Leech\to \Leechminus\dual.
\]
Since $\Lts$ and $\Leech$ are unimodular, 
these projections are surjective.
%
%
%
To simplify the notation, for $\lambda\in \Leech$, 
we put
\[
\lambdaplus:=\prLeechplus(\lambda),
\qquad 
\lambdaminus:=\prLeechminus(\lambda).
\]
Since $\Leechplus$ and $ \Leechminus$ are orthogonal, and 
the component $\tau$ of 
the fixed involution $\gamma=(\vep, \tau)$ belongs to $ \Leechminus$, 
we have 
\[
\lambda^2=\lambdaplus^2+\lambdaminus^2,
\quad
\intfLeech{\lambda, \tau}=\intfLeech{\lambdaminus, \tau}, 
\quad
\left(\lambdaplus +\frac{\tau}{2} \right)^2=\lambdaplus^2+\frac{\tau^2}{4}.
\]
For $\lambda\in \Leech$, we put
\begin{equation}\label{eq:defmlambda}
\renewcommand{\arraystretch}{2}
\begin{array}{lll}
m(\lambda) &:=& \prLplus(\LeechRoot{\lambda})=\dfrac{\;1\;}{2}(\LeechRoot{\lambda}+
\LeechRoot{\lambda^\vep+\tau})\\
&=&
\left(\dfrac{\lambda^2}{2}-\dfrac{\intfLeech{\lambda, \tau}}{2}+\dfrac{\tau^2}{4}-1, 
\;\; 1,\;\; \lambdaplus +\dfrac{\tau}{2}\right)\\
&=&
\left(
\dfrac{\lambdaplus^2}{2}+\dfrac{1}{\,2\,}\left(\lambdaminus-\dfrac{\tau}{\,2\,}\right)^2+\dfrac{\tau^2}{8}-1, 
\;\; 1,\;\; \lambdaplus +\dfrac{\tau}{2}\right).
\end{array}
\end{equation}
For $v\in \Leechminus\dual$,
we put
\[
n(v):=\left(v-\dfrac{\,\tau\,}{2}\right)^2.
\]
Then we have
\begin{equation}\label{eq:mlmlsprime}
 \intfL{m(\lambda), m(\lambda\sprime)} =\frac{1}{\,2\,}
 \left(\,
(\lambdaplus - \lambdaplus\sprime )^2
+
n(\lambdaminus)
+
n(\lambdaminus\sprime)
\;-\;4 
\,\right).
\end{equation}
In particular, we have 
\[
\intfL{m(\lambda), m(\lambda)}=n(\lambdaminus)-2.
\]
We also put 
\[
\renewcommand{\arraystretch}{1.3}
\begin{array}{lll}
\tilwalls(D_0) &:= & \set{\lambda\in \Leech}{\textrm{$(\LeechRoot{\lambda})\sperp \cap D_0$ is a wall of $D_0$}}, \\
\MMM(D_0) &:=& \set{\lambda\in \Leech}{\intfL{m(\lambda), m(\lambda)}<0}=\set{\lambda\in \Leech}{n(\lambdaminus)<2}.
\end{array}
\]
Then we have 
\begin{equation}\label{eq:tilWWWMMM}
\tilwalls(D_0)\subset \MMM(D_0).
\end{equation}
Note that the stabilizer $\Cen(\dotinf, \gamma)$ of the Conway chamber $\ConCham(\weyl_0)$
acts on 
$\tilwalls(D_0)$ and $\MMM(D_0)$.
Since the restriction homomorphism $\tilg\mapsto \tilg|_{\Lplus}$ maps $\Cen(\dotinf, \gamma)$ to $\Stab_{\vGam}(D_0)$, 
it is enough to show that 
 the number of orbits of the action of $\Cen(\dotinf, \gamma)$ on $ \MMM(D_0)$
is finite.
\par
We put
\begin{equation}\label{eq:Xi}
\Xitau:=\bigset{v\in \Leechminus\dual }{n(v)<2},
\end{equation}
which is a finite subset of $\Leechminus\dual$.
Since $\prLeechminus\colon \Leech \to \Leechminus\dual$ is surjective,
the projection $\MMM(D_0) \to \Xitau$ is also surjective.
Moreover, by translation,
the group $\Leechplus$ acts 
on each fiber of $\MMM(D_0) \to \Xitau$ 
transitively and freely.
Therefore we have 
\[
\MMM(D_0)/\Leechplus \;\cong \; \Xitau. 
\]
Recall from~\eqref{eq:exactCen} that
we have an exact sequence
\begin{equation}\label{eq:exactCenagain}
0\;\;\longrightarrow\;\;\Leechplus\;\; \longrightarrow \;\; 
\Cen(\dotinf, \gamma)
\;\; \longrightarrow \;\; 
\Sigma_{\vep}([\tau])
\;\; \longrightarrow \;\; 1,
\end{equation}
where $\Sigma_{\vep}([\tau])$ is the subgroup of $\Cen(\dotz, \vep)$
consisting of $h\in \Cen(\dotz, \vep)$ such that 
there exists an element $\sigma\in \Leech$ such that $\tau-\tau^h=\sigma-\sigma^{\vep}$.
Therefore the action 
of $\Cen(\dotinf, \gamma)$ on $\MMM(D_0)$ induces an action of 
$\Sigma_{\vep}([\tau])$ on $\Xitau$. 
This induced action is described easily by introducing 
\begin{equation}\label{eq:tilXi}
\tilXitau:=\Xitau -\dfrac{\tau}{\,2\,}
=\bigset{u\in \Leechminus\dual\tensor\QQ}{\;u^2<2, \;\; u+\dfrac{\tau}{\,2\,}\in \Leechminus\dual}.
\end{equation}
We show that, via the natural action of $\Cen(\dotz, \vep)$ on $\Leechminus$,
the subgroup 
$\Sigma_{\vep}([\tau])$ of $\Cen(\dotz, \vep)$
preserves the subset $\tilXitau$ of $\Leechminus\dual\tensor\QQ$,
and that 
the induced action of 
$\Sigma_{\vep}([\tau])$ on $\Xitau$
is equal to the natural action of $\Sigma_{\vep}([\tau])$ on $\tilXitau$
via the 
bijection $\tilXitau\cong \Xitau$ given by 
$x\mapsto x+\tau/2$.
Suppose that $h\in \Sigma_{\vep}([\tau])$ and $v\in \Xitau$.
We choose a vector $\sigma\in \Leech$ such that $\tau-\tau^h=\sigma-\sigma^\vep=2\sigma_{-}$.
Then $(h, \sigma)\in \Cen(\dotinf, \gamma)$ is a lift of $h$ in~\eqref{eq:exactCenagain}.
We also choose a lift $\lambda\in \Leech$ of $v$,
that is, we have $v=\lambda\sb{-}$.
Then the induced action of $h$ on $\Xitau$ maps $v$ to $(\lambda^h+\sigma)\sb{-}$.
This action is compatible with the natural action $x\mapsto x^h$ on $\tilXitau$, because we have 
\[
\left(v-\dfrac{\tau}{\,2\,}\right)^h=(\lambda^h+\sigma)_{-}-\dfrac{\tau}{\,2\,}\;\;\in\;\;\tilXitau.
\]
Therefore we can explicitly calculate  the set 
\begin{equation}\label{eq:orbitss}
 \MMM(D_0)/\Cen(\dotinf, \gamma)= \MMM(D_0)/\Leechplus/\Sigma_{g}([\tau])\;\;\cong\;\; \Xitau/\Sigma_{g}([\tau]). 
\end{equation}
Since $\Xitau$ is finite, 
we see that 
$\tilwalls(D_0)/\Cen(\dotinf, \gamma)$ is finite. 
Hence  $\walls(D_0)/\Stab_{\vGam}(D_0)$ is also finite. 
\end{proof}
%
%
We consider the problems when 
the equality $\MMM(D_0)=\tilwalls (D_0)$ holds and 
what are the fibers of the natural mapping $\tilwalls (D_0)\to \walls(D_0)$.
\begin{lemma}\label{lem:mmequal}
For $\lambda, \lambda\sprime\in \MMM(D_0)$, we have 
\[
\renewcommand{\arraystretch}{1.2}
\begin{array}{cl}
&(m(\lambda))\sperp=(m(\lambda\sprime))\sperp\\
\Longleftrightarrow & m(\lambda)=m(\lambda\sprime) \\
\Longleftrightarrow & \lambdaplus=\lambdaplus\sprime\;\;\textrm{and}\;\; 
n(\lambdaminus)=n(\lambdaminus\sprime).
\end{array}
\]
\end{lemma}
\begin{proof}
These follow from the definition~\eqref{eq:defmlambda}.
\end{proof}
First we recall elementary results in hyperbolic geometry.
Let $L$ be a hyperbolic lattice, and $\PPP$ a positive cone of $L$.
For a vector $v\in L\tensor\RR$ with $\intf{v, v}<0$,
let $(v)\sperp$ denote the hyperplane of $\PPP$ defined by $\intf{x, v}=0$.
Let $v\in L\tensor\RR$ be a vector with $\intf{v, v}<0$, and 
$p$ a point of $\PPP$.
The \emph{foot of the perpendicular from $p$ to the hyperplane $(v)\sperp$ }
is given by
\[
q:=p-\frac{\intf{p, v}}{\intf{v,v}}\, v.
\]
Then we have $\intf{p, q}>0$, and 
by the Lorentzian Cauchy--Schwarz inequality,
we have $\intf{q, q}>0$.
Therefore we have 
$q\in \PPP$.
Let $v\sprime\in L\tensor\RR$ be another vector with $\intf{v\sprime, v\sprime}<0$.
We assume that $\intf{p, v}>0$ and $\intf{p, v\sprime}>0$.
Then we have
\begin{equation}\label{eq:ifvvsprime}
\intf{v, v\sprime}\ge 0 \;\;\Longrightarrow \;\; \intf{q, v\sprime}=
\intf{p, v\sprime}-\dfrac{\intf{p, v}}{\intf{v,v}} \intf{v, v\sprime}>0,
\end{equation}
that is, if $\intf{v, v\sprime}\ge 0$, 
then the foot $q\in (v)\sperp$ is located in the same half-space as $p$ with respect to the hyperplane 
$(v\sprime)\sperp$.
\begin{lemma}\label{lem:pairsofMMM}
Suppose that we have
$\intfL{m(\lambda), m(\lambda\sprime)}\ge 0$
for every  pair $\lambda, \lambda\sprime\in  \MMM(D_0)$
satisfying  $m(\lambda)\ne m(\lambda\sprime)$.
Then we have 
$\MMM(D_0)=\tilwalls (D_0)$.
\end{lemma}
\begin{proof}
Since $\MMM(D_0)\supset \tilwalls (D_0)$,
 we have 
 \[
 D_0=\set{x\in \PPPplus}{\intfL{x, m(\lambda)}\ge 0\;\; \textrm{for any} \;\; \lambda\in \MMM(D_0)}.
\]
Hence a vector $\lambda\in \MMM(D_0)$ belongs to $\tilwalls (D_0)$
if and only if 
$(m(\lambda))\sperp \cap D_0$ contains a non-empty open subset of $(m(\lambda))\sperp $,
or equivalently,
if and only if 
there exists a point $q\in (m(\lambda))\sperp$
such that $\intfL{q, m(\lambda\sprime)}> 0$ holds for any $\lambda\sprime\in \MMM(D_0)$
with $m(\lambda)\ne m(\lambda\sprime)$.
Considering the foot $q$ of 
the perpendicular from an interior point $p$ of $D_0$
to the hyperplane $(m(\lambda))\sperp$,
we obtain the proof from~\eqref{eq:ifvvsprime}.
\end{proof}
For a pair $\xi, \xi\sprime\in \Xitau$,
we put
\begin{equation}\label{eq:diffellset}
\Delta(\xi, \xi\sprime):=
\set{l\in \Leechplus\dual}{l+\xi-\xi\sprime\in \Leech},
\end{equation}
on which $\Leechplus= \Leechplus\dual\cap\Leech$ acts freely and transitively.
We then put 
\begin{equation}\label{eq:diffell}
\ell(\xi, \xi\sprime):=\min\set{l^2}{l\in \Delta(\xi, \xi\sprime)}.
\end{equation}
Then we obtain the following:
\begin{lemma}\label{lem:MMMtilwalls}
Suppose that, for any pair $\xi, \xi\sprime\in \Xitau$,
we have
\begin{equation}\label{eq:nnell}
(\;n(\xi)\ne n(\xi\sprime) \; \textrm{\rm or}\; \;\ell(\xi, \xi\sprime)\ne 0 \;)\;\;\Longrightarrow\;\; \ell(\xi, \xi\sprime)+n(\xi)+n(\xi\sprime)\ge 4.
\end{equation}
Then $\MMM(D_0)=\tilwalls (D_0)$ holds.
\end{lemma}
\begin{proof}
Assume that $\MMM(D_0)\supsetneqq\tilwalls (D_0)$.
By Lemma~\ref{lem:pairsofMMM},
there exists a pair $\lambda, \lambda\sprime$ of elements of $ \MMM(D_0)$
such that $m(\lambda)\ne m(\lambda\sprime)$ and $ \intfL{m(\lambda), m(\lambda\sprime)} <0$.
We will show that the negation of~\eqref{eq:nnell} holds for
$\xi:=\lambdaminus$ and $\xi\sprime:=\lambdaminus\sprime$.
We put $l:=\lambdaplus-\lambdaplus\sprime\in \Leechplus\dual$,
which is an element of $\Delta(\xi, \xi\sprime)$,
and hence $\ell(\xi, \xi\sprime)\le l^2$.
By~\eqref{eq:mlmlsprime},
the assumption $ \intfL{m(\lambda), m(\lambda\sprime)} <0$ implies 
\begin{equation*}\label{eq:lnn}
l^2+n(\xi)+n(\xi\sprime)<4.
\end{equation*}
Thus the right-hand side of $\Longrightarrow$ in~\eqref{eq:nnell} is false.
Moreover, we have $l^2<4$.
Suppose that $n(\xi)=n(\xi\sprime)$.
We show that $\ell(\xi, \xi\sprime)\ne 0$.
By Lemma~\ref{lem:mmequal},
the assumption $m(\lambda)\ne m(\lambda\sprime)$ implies $l\ne 0$.
If $\ell(\xi, \xi\sprime)= 0$, then 
$\xi-\xi\sprime\in \Leech$ holds, and hence 
$ \Delta(\xi, \xi\sprime)\subset  \Leech$.
Then $l\in\Delta(\xi, \xi\sprime)$ and $l\ne 0$ 
imply $l^2\ge 4$, a contradiction.
Thus  the left-hand side of $\Longrightarrow$ in~\eqref{eq:nnell} is true.
\end{proof}
%
%
\begin{remark}\label{rem:MMMtilwalls}
Since $\Xitau$ is finite,
checking the condition~\eqref{eq:nnell} for all pairs $\xi, \xi\sprime$
 can be carried out by a brute-force method.
 More precisely, 
choosing an element  $\delta\in \Delta(\xi, \xi\sprime)$
and fixing a bijection $\Delta(\xi, \xi\sprime)\cong \Leechplus$ by $x\mapsto x-\delta$,
the value  $\ell(\xi, \xi\sprime)$ is calculated as the minimum of
an inhomogeneous quadratic form on the free $\ZZ$-module $\Leechplus$.
\end{remark}
We write the natural mapping
$\tilwalls(D_0)\to \walls(D_0)$
by $\mu$, that is, 
\[
\mu(\lambda):=(m(\lambda))\sperp\cap D_0
\]
for $\lambda\in \tilwalls(D_0)$.
Note that the fiber of $\mu$ over a wall $w$
of $D_0$ is equal to the set of walls $(\LeechRoot{\lambda})\sperp\cap \ConChamz$
of $\ConChamz$ whose intersection with $\PPPplus$ is $w$.
By Lemma~\ref{lem:mmequal}, we have the following:
\begin{lemma}\label{lem:wallfiber}
Suppose that $\MMM(D_0)=\tilwalls (D_0)$ holds.
Then the fiber $\mu\inv(\mu(\lambda))$ 
of 
$\mu\colon \tilwalls(D_0)\to \walls(D_0)$
over a wall $(m(\lambda))\sperp\cap D_0$ of $D_0$
is equal to
\[
\set{\lambda\sprime\in \Leech}%
{\lambdaplus\sprime=\lambdaplus, \;\; \lambdaminus\sprime\in \Xitau, \;\; n(\lambdaminus\sprime)=n(\lambdaminus)},
\]
which is in one-to-one correspondence with the subset 
\[
\set{\xi\sprime\in \Xitau}{\xi\sprime-\lambdaminus\in \Leech, \;\; n(\xi\sprime)=n(\lambdaminus)}.
\]
of $\Xitau$ via $\lambda\sprime \mapsto \xi\sprime=\lambdaminus\sprime$.
\qed
\end{lemma}
In the next three sections, we apply these lemmas to the cases $\gamma=(\vep_8, 0), (\vep_{12}, 0)$,
and $\gamma=(\vep_{16}, \tau)$ for $\tau=0,\tauI,\tauII$.
More precisely,
we carry out the following computations:
\begin{itemize}
\item 
We calculate the index $[\OG(\Lplus, \PPPplus): \vGam]$ 
by Proposition~\ref{prop:vGamindex},
where $\vGam$ is the group defined in~\eqref{eq:vGamdef}.
 For this purpose, 
we compute the index of 
the image of $\etaminus\colon \OG(\Lminus)\to \OG(\qminus)$, and 
show that 
$\etaplus\colon \OG(\Lplus, \PPPplus)\to \OG(\qplus)$
is surjective.
Since $\Lminus$ is negative-definite,
we can use the method described in Section~\ref{subsec:computational}
for the computation of the index of  $\Image \etaminus$.
For the surjectivity of $\etaplus$, 
the general theory in~\cite[Chapter VIII]{MirandaMorrison2009} can be applied.
However,  we instead give simpler arguments for each case.
\item 
We compute the finite set $\Xitau$
defined in~\eqref{eq:tilXi}.
%
%
We then verify that $\MMM(D_0)=\tilwalls (D_0)$ holds
by Lemma~\ref{lem:MMMtilwalls}.
See Remark~\ref{rem:MMMtilwalls}.
Next we 
calculate a finite generating set of 
the subgroup $\Sigma_{\vep}([\tau])$ of $\Cen(\dotz, \vep)$,
and 
compute the orbit decomposition of $\Xi_{\tau}$ under the action of $\Sigma_{\vep}([\tau])$.
Thus we obtain 
the $\Stab_{\vGam}(D_0)$-orbits in $\walls(D_0)$.
\item
For each $\Stab_{\vGam}(D_0)$-orbit in $\walls(D_0)$,
we compute by Lemma~\ref{lem:wallfiber} the fiber $\mu\inv(\mu(\lambda))$ of $\mu\colon \tilwalls(D_0)\to \walls(D_0)$
over a representative wall $\mu(\lambda)$ of the orbit.
\end{itemize}
The results are summarized in Tables~\ref{table:Lpmqpm} and~\ref{table:Xis}.
%
\begin{table}
\[
\renewcommand{\arraystretch}{1.2}
\setlength{\arraycolsep}{10pt}
\begin{array}{cccr}
\gamma & \Lplus &\Lminus& |\OG(\Lminus)| \\
\hline 
(\vep_{8}, 0) &\;\; U\oplus \Lamsxtn(-1) &E_8(-2) & 696729600 \\
(\vep_{12}, 0) &\;\; U\oplus \Sigmatwlv(-2) &\Sigmatwlv(-2) &  980995276800\\
(\vep_{16}, 0) &\;\; U\oplus E_8(-2) &\Lamsxtn(-1) & 89181388800 \\
(\vep_{16},  \tauI) & \typeILat_{1, 9} (2) & \textrm{See Proposition~\ref{prop:involL}}& 743178240 \\
(\vep_{16}, \tauII) & \typeIILat_{1, 9} (2) & \textrm{See Proposition~\ref{prop:involL}} & 21139292160 
\end{array}
\]
\vskip .5mm
\[
\renewcommand{\arraystretch}{1.2}
\setlength{\arraycolsep}{10pt}
\begin{array}{ccrr}
\gamma & \qplus\cong \qminus & |\OG(\qplus)| = |\OG(\qminus)| & \textrm{index of $\Image \etaminus$}\\
\hline 
(\vep_{8}, 0) & u^{\oplus 4} & 348364800 & 1\\
(\vep_{12}, 0) & a^{\oplus 12} &51231497335603200  &104448 \\
(\vep_{16}, 0) & u^{\oplus 4} & 348364800 &2 \\
(\vep_{16},  \tauI) & (a\oplus b)^{\oplus 5} &89181388800 &240 \\
( \tauII )& u^{\oplus 5} &46998591897600 &71145
\end{array}
\]
\caption{Discriminant form and sizes of groups}\label{table:Lpmqpm}
\end{table}
\begin{table}
\[
\setlength{\arraycolsep}{10pt}
\renewcommand{\arraystretch}{1.4}
\begin{array}{crrcc}
\gamma &\textrm{index of $\vGam$} &\textrm{orbits in $\Xitau$} & \xi^2\textrm{\;\,for\;} \xi\in \tilXitau&\mu\inv(\mu(\lambda))\\
\hline 
(\vep_8, 0) &1 & 241= 1 & 0 & f_0 \\
  & & +240 & 1 & F_0 \\
\hline 
(\vep_{12}, 0) &104448 & 2313= 1 & 0 & f_0 \\
  & & +264 & 1 & F_0 \\
  & & +2048 & 3/2 & F_1 \\
 \hline 
(\vep_{16}, 0) &2& 1 & 0 & f_0 \\
\hline 
(\vep_{16},\tauI) &240& 512 & 3/2 & F_1 \\
\hline 
(\vep_{16},\tauII) &71145 & 32 & 1 & F_0 \\  
\end{array}
\]
\vskip 3mm
\parbox{12cm}{
Explanation of the column labelled   $\mu\inv(\mu(\lambda))$: \\
$f_0$ means that $\mu\inv(\mu(\lambda))$ is a singleton $\{\lambda\}$,
\\
$ F_0$ means that $\mu\inv(\mu(\lambda))=\{\lambda, \lambda\sprime\}$ with $\intf{r(\lambda), r(\lambda\sprime)}=0$, 
\\
$ F_1$ means that $\mu\inv(\mu(\lambda))=\{\lambda, \lambda\sprime\}$ with $\intf{r(\lambda), r(\lambda\sprime)}=1$.
}
\vskip 5mm
\caption{Index $[\OG(\Lplus, \PPPplus): \vGam]$ and orbit decomposition of walls}\label{table:Xis}
\end{table}
\subsection{Involution $\gamma=(\vep_8, 0)$}\label{subsec:walls8}
In this case, we have
\[
\begin{array}{ll}
\Leechplus\cong \Lamsxtn, \quad &\Leechminus\cong \Laminated_{8}\cong E_8(2),\\
\Lplus\cong U\oplus \Lamsxtn(-1), \quad &\Lminus\cong E_8(-2),
\end{array}
\]
and the discriminant forms $\qplus\cong \qminus$ of these lattices are 
isomorphic to $q_8=u^{\oplus 4}$.
We have proved the surjectivity of $\etaminus$ in Section~\ref{subsec:vep8}.
Hence, by~\eqref{eq:groupfibprodP}, we obtain 
\[
[\OG(\Lplus, \PPPplus): \vGam]=1.
\]
%
%
Since $\tau=0$,
we have $\Xitau=\tilXitau $ and $\Sigma_{\vep}([\tau])=\Cen(\dotz, \vep_{8})$.
Since $E_8$ is unimodular, we have $\Leechminus\dual=(1/2)E_8(2)$.
Consequently, we obtain 
\[
\Xitau\;=\;\tilXitau \; =\; \{0\}\;\cup\;\set{\xi\in \Leechminus\dual}{\xi^2=1},
\]
which consists of $1+240$ vectors.
We prove~\eqref{eq:nnell} holds for all pairs $\xi, \xi\sprime\in \Xitau$
by the method in Remark~\ref{rem:MMMtilwalls}.
Hence 
$\MMM(D_0)=\tilwalls(D_0)$ holds.
%
%
%
The action of $\Sigma_{\vep}([\tau])$ on $\Xitau$
is given by the projection 
$\Cen(\dotz, \vep_{8})\to \OG(\Leechminus)$.
By the diagram~\eqref{eq:groupfibprodCenvep}, and results in Section~\ref{subsec:vep8}
(in particular, the fact~\eqref{eq:Imageeta16}), 
the image of the projection $\Cen(\dotz, \vep_{8})\to \OG(\Leechminus)$ is 
the pullback of $\OG^{+}_8(2)\subset \OG(q_8)$ by the natural homomorphism 
$\eta_8\colon W(E_8)\to \OG(q_8)$.
Thus we have 
\[
\Image(\Cen(\dotz, \vep_{8})\to \OG(\Leechminus) )=\set{g\in W(E_8)}{\det g=1}=W(E_8)\sprime,
\]
where $W(E_8)\sprime$ is the derived group of $W(E_8)$.
The action of $W(E_8)\sprime$ on $\Xitau$ has exactly two orbits:
one is 
$\{0\}$ and the other is $\set{\xi\in \Leechminus\dual}{\xi^2=1}$.
Hence the action of $\Stab_{\vGam}(D_0)$ on $\walls(D_0)$ 
has also two orbits.
\par
The orbit $o_1$ corresponding to $\{0\}$
consists of walls $\mu(\lambda)$,
where $\lambda$ runs through $\Leechplus$,
and the fiber of $\mu\colon \tilwalls(D_0)\to\walls(D_0)$ 
over $\mu(\lambda)\in o_1$ is a singleton $\{\lambda\}$.
The orbit $o_{240}$ corresponding to $\set{\xi\in \Leechminus\dual}{\xi^2=1}$
consists of walls $\mu(\lambda)$,
where $\lambda\in \Leech$ runs through
vectors of $\Leech$ satisfying $\lambdaminus^2=1$.
The fiber of $\mu\colon \tilwalls(D_0)\to\walls(D_0)$ 
over $\mu(\lambda)\in o_{240}$ consists 
of two vectors
$\lambda=\lambdaplus+\lambdaminus$ and $\lambda\sprime=\lambdaplus-\lambdaminus$.
The corresponding pair of Leech roots 
$\LeechRoot{\lambda}$ and $\LeechRoot{\lambda\sprime}$ satisfies
$\intfL{\LeechRoot{\lambda}, \LeechRoot{\lambda\sprime}}=0$.
%
%
\subsection{Involution $\gamma=(\vep_{12}, 0)$}\label{subsec:walls12}
%
%
In this case, we have
\[
\begin{array}{ll}
\Leechplus\cong \Sigmatwlv(2), \quad &\Leechminus\cong \Sigmatwlv(2),\\
\Lplus\cong U\oplus \Sigmatwlv(-2), \quad &\Lminus\cong \Sigmatwlv(-2), 
\end{array}
\]
and their discriminant forms $\qplus\cong \qminus$ are all isomorphic to
$q_{12}:=a^{\oplus 12}$.
%
%
%
%
%
%
\begin{proposition}
The natural homomorphism 
$\etaplus\colon \OG(\Lplus, \PPPplus) \to \OG(\qplus)$ is surjective.
Hence the index 
$[\OG(\Lplus, \PPPplus): \vGam]$ is equal to 
$104448$.
\end{proposition}
\begin{proof}
We have an inclusion $\OG(\Sigmatwlv)\inj \OG(\Lplus, \PPPplus)$
given by $g\mapsto \id_{U}\oplus g$,
and, by the method in Section~\ref{subsec:computational}, 
we observe that 
 the image $\etaplus(\OG(\Sigmatwlv))$ of this finite subgroup 
by $\etaplus$ is 
of size $490497638400=2^{10}\cdot 12!=|\OG(\Sigmatwlv)|/2$.
 \par
 We use the construction of $\Sigmatwlv$ in Section~\ref{subsec:vep12}.
 Let $v$ be the vector $2e_1+\cdots +2e_6$ of $\Sigmatwlv$ with norm $6$.
 We consider the vectors
 \[
 \vect{u}:=2\weyl_0+ 3\weyl_1+v,
 \quad 
 \vect{u}\sprime:=3\weyl_0+ 2\weyl_1+v,
 \]
 of $\Lplus=U\oplus \Sigmatwlv(-2)$ with norm $0$,
 where $\weyl_0$ and $\weyl_1$ are the chosen basis of $U$. 
Since 
$\intf{\vect{u}, \vect{u}\sprime}=1$, these two vectors 
generate a hyperbolic plane $U\sprime$ in $\Lplus$.
 The orthogonal complement $(U\sprime)\sperp$ of $U\sprime$ in $\Lplus$
 is of the form $\varSigma\sprime(-2)$,
where $\varSigma\sprime$ is a positive-definite odd unimodular lattice
with no vectors of norm $1$.
Therefore $\varSigma\sprime$ is isomorphic to $\Sigmatwlv$.
We calculate the frame of $\varSigma\sprime$ 
by the method given in the proof of Lemma~\ref{lem:frame},
and find an isomorphism $g\sprime \colon \Sigmatwlv\cong \varSigma\sprime$.
Combining $\weyl_0\mapsto \vect{u}$, $\weyl_1\mapsto \vect{u}\sprime$,
and $g\sprime \colon \Sigmatwlv\cong \varSigma\sprime$,
we obtain an automorphism $g\colon \Lplus\to\Lplus$.
Since ${\weyl_0}^g= \vect{u}$, 
the action of $g$ preserves $\PPPplus$.
Then $g\in \OG(\Lplus, \PPPplus)$
 is not contained in the subgroup $\OG(\Sigmatwlv)$.
We calculate the size of the subgroup of $\OG(\qplus)$
generated by $\etaplus(\OG(\Sigmatwlv))$ and $\etaplus (g)$
by the method given in Section~\ref{subsec:computational}, 
and confirm that 
this size is equal to $|\OG(\qplus)|=51231497335603200$.
Therefore $\etaplus$ is surjective.
\par
The computation concerning $\etaminus$ has already been  done in 
Section~\ref{subsec:vep12},
and we know that the image of 
the natural homomorphism 
$\etaminus\colon \OG(\Lminus) \to \OG(\qminus)$
is of size 
$|\OG(\Lminus)|/2=490497638400$.
By Proposition~\ref{prop:vGamindex},
we obtain 
\[
[\OG(\Lplus, \PPPplus): \vGam]
= \frac{ |\OG(q_{12})|}{\; |\Image(\etaminus)|\;}
=
\frac{51231497335603200}{490497638400}
=
104448.
\]
This proves the proposition.
\end{proof}
Since $\tau=0$,
we have $\Xitau=\tilXitau$ and $\Sigma_{\vep}([\tau])=\Cen(\dotz, \vep_{12})$.
Since $\Sigmatwlv$ is unimodular,
we have $\Leechminus\dual= (1/2)\Sigmatwlv(2)$.
Therefore we have 
\[
\Xitau=\tilXitau=\set{\xi\in \Leechminus\dual}{\xi^2=0, \;\textrm{or}\; \; \xi^2=1, \;\textrm{or}\; \; \xi^2=3/2},
\]
which is of size $1+264+2048$.
Using the construction of $\Sigmatwlv$ in Section~\ref{subsec:vep12} and 
identifying $\Leechminus\dual$ with $ (1/2)\Sigmatwlv(2)$,
we see that 
\[
\renewcommand{\arraystretch}{1.4}
\begin{array}{lcl}
\set{\xi\in \Leechminus\dual}{\xi^2=1} &=& \set{\pm e_i\pm e_j}{i\ne j}, \\
\set{\xi\in \Leechminus\dual}{\xi^2=3/2} &=& \set{(\sum s_i e_i)/2}{\textrm{$s_i\in \{1, -1\}$ and $\prod s_i=1$}}. \\
\end{array}
\]
%
%
We prove~\eqref{eq:nnell} holds for all pairs $\xi, \xi\sprime\in \Xitau$
by the method in Remark~\ref{rem:MMMtilwalls}.
Hence 
$\MMM(D_0)=\tilwalls(D_0)$ holds.
\par 
Recall that $C_{12}\in \Golay$ is the codeword such that $\vep(C_{12})=\vep_{12}$. 
From a generating set of $M_{24}$ given in~\cite[Chapter 10, Section 2]{theCSbook}, 
we compute a generating set of 
the stabilizer $M_{12}$ of $C_{12}$ in $M_{24}$.
Hence, by~\eqref{eq:M12},
we can compute the action of $\Cen(\dotz, \vep_{12})$ on $\Leech$ explicitly.
We confirm that this action decomposes $\Xitau=\tilXitau$
into three orbits of size $1$, $264$, $2048$.
Hence the action of $\Stab_{\vGam}(D_0)$ on $\walls(D_0)$ 
has also three orbits,
which are described as follows.
\begin{itemize}
\item[($o_{1}$)]
The orbit $o_1$ corresponding to $\{0\}$
consists of walls $\mu(\lambda)$,
where $\lambda$ runs through $\Leechplus$.
The fiber of $\mu\colon \tilwalls(D_0)\to\walls(D_0)$ 
over a wall $\mu(\lambda)\in o_1$ is a singleton $\{\lambda\}$.
\item[($o_{264}$)]
The orbit $o_{264}$ corresponding to $\set{\xi\in \Leechminus\dual}{\xi^2=1}$
consists of $\mu(\lambda)$,
where $\lambda\in \Leech$ runs through
vectors of $\Leech$ satisfying $\lambdaminus^2=1$.
The fiber $\mu\inv(\mu(\lambda))$  of $\mu$ 
over $\mu(\lambda)\in o_{264}$ consists 
of two vectors
$\lambda=\lambdaplus+\lambdaminus$ and $\lambda\sprime=\lambdaplus-\lambdaminus$.
The corresponding pair of Leech roots 
$\LeechRoot{\lambda}$, $\LeechRoot{\lambda\sprime}$ 
satisfies
$\intfL{\LeechRoot{\lambda}, \LeechRoot{\lambda\sprime}}=0$.
\item[($o_{2048}$)]
The orbit $o_{2048}$ corresponding to $\set{\xi\in \Leechminus\dual}{\xi^2=3/2}$
consists of $\mu(\lambda)$,
where $\lambda\in \Leech$ runs through
vectors of $\Leech$ satisfying $\lambdaminus^2=3/2$.
For a wall $\mu(\lambda)\in o_{2048}$, the fiber
 $\mu\inv(\mu(\lambda))$ consists of
$\lambda=\lambdaplus+\lambdaminus$ and $\lambda\sprime=\lambdaplus-\lambdaminus$, and 
the corresponding pair of Leech roots 
 satisfies
$\intfL{\LeechRoot{\lambda}, \LeechRoot{\lambda\sprime}}=1$.

\end{itemize}
%
 %
\subsection{Involutions $\gamma=(\vep_{16}, \tau)$}\label{subsec:walls16}
%
%
%
%
%
%
%
We have 
\[
\Leechplus\cong\Laminated_{8}\cong E_8(2), \quad \Leechminus\cong \Lamsxtn.
\]
The isomorphism classes of the lattices $L_{\pm}$ and 
the finite quadratic forms $q_{\pm}$ are described in Table~\ref{table:Lpmqpm}.
The indices of the images of $\etaminus$ are also given there.
\begin{proposition}\label{prop:etaplussurj16}
The natural homomorphism $\etaplus\colon \OG(\Lplus, \PPPplus)\to \OG(\qplus)$ is surjective.
Hence the index of $\vGam$ in $\OG(\Lplus, \PPPplus)$ is
\[
[\OG(\Lplus, \PPPplus): \vGam]=
\begin{cases}
2 & \textrm{if $\tau=0$}, \\
240 & \textrm{if $\tau= \tauI$}, \\
71145 & \textrm{if $\tau= \tauII$}.
\end{cases}
\]
\end{proposition}
\begin{proof}
By the method described in Section~\ref{subsec:computational},
we calculate the values of $|\OG(\qplus)|$, 
which are given in Table~\ref{table:Lpmqpm}.
It is enough to find a finite subset $S$ 
of $\OG(\Lplus, \PPPplus)$ such that 
the subgroup $\angs{\etaplus(S)}$ of $\OG(\qplus)$
generated by $\etaplus(S)$ is of size equal to $|\OG(\qplus)|$.
\par
Suppose that $\tau=0$.
In this case,
we have $\Lplus\cong U\oplus E_8(-2)$.
We have a natural embedding $\OG(E_8)=W(E_8)\inj \OG(\Lplus, \PPPplus)$
given by $g\mapsto \id_U \oplus g$.
Then 
the set $S$ consisting of the standard $8$ reflections in $W(E_8)$ 
satisfies $|\angs{\etaplus(S)}|=|\OG(\qplus)|$.
\par
Suppose that $\tau= \tauI$.
In this case,
we have $\Lplus\cong \typeILat_{1, 9}(2)$.
By Vinberg’s algorithm~\cite{Vinberg1985}, 
we obtain a basis of the odd unimodular hyperbolic lattice $\typeILat_{1,9}$ 
whose Coxeter–Dynkin diagram is shown in the upper part of Figure~\ref{fig:CDdiagrams}. 
(In this diagram, 
the black node represents a $(-1)$-vector, and 
the dashed edge indicates that 
$\intf{\hbox{\large{$\bullet, \circ$}}}=-1$.)
The reflections 
with respect to these $10$ vectors
belongs to $\OG(\Lplus, \PPPplus)$, and 
 the set $S$ consisting of these  reflections
satisfies $|\angs{\etaplus(S)}|=|\OG(\qplus)|$.
\par
Suppose that $\tau= \tauII$.
In this case,
we have $\Lplus\cong \typeIILat_{1, 9}(2)$,
which 
has a basis 
whose Coxeter–Dynkin diagram is shown in the lower part of Figure~\ref{fig:CDdiagrams}. 
Then the set $S$ of the
 $10$ reflections corresponding to the nodes in this diagram
satisfies $|\angs{\etaplus(S)}|=|\OG(\qplus)|$.
\begin{figure}%
\begin{tikzpicture}[scale=0.8, every node/.style={circle, draw, inner sep=2pt}]
\node (a1) at (0,0) {};
\node (a2) at (1.5,0) {};
\node (a3) at (3,0) {};
\node (a4) at (4.5,0) {};
\node (a5) at (6,0) {};
\node (a6) at (7.5,0) {};
\node (a7) at (9,0) {};
\node (a8) at (10.5,0) {};
\node (a9) at (3,1.5) {}; 
\node[fill=black] (a10) at (1.5,1.5) {}; 
\node[draw=none] (phantom) at (12,0) {};
%
\draw (a1)--(a2)--(a3)--(a4)--(a5)--(a6)--(a7)--(a8);
\draw (a3)--(a9); 
\draw (a2)--(a10); 
\draw[dashed] (a9)--(a10);
\end{tikzpicture}

\vskip .8cm
\begin{tikzpicture}[scale=0.8, every node/.style={circle, draw, inner sep=2pt}]
%
\node (a2) at (0,0) {};
\node (a3) at (1.5,0) {};
\node (a4) at (3,0) {};
\node (a5) at (4.5,0) {};
\node (a6) at (6,0) {};
\node (a7) at (7.5,0) {};
\node (a8) at (9,0) {};
\node (a9) at (10.5,0) {};
\node (a10) at (12,0) {};
\node (a1) at (3,1.5) {};
%
\draw (a2)--(a3)--(a4)--(a5)--(a6)--(a7)--(a8)--(a9)--(a10);
\draw (a4)--(a1);
\end{tikzpicture}
\caption{Coxeter-Dynkin diagrams for $\typeILat_{1, 9}$ and $\typeIILat_{1, 9}$}\label{fig:CDdiagrams}
\end{figure}
%
%
\par
By Proposition~\ref{prop:vGamindex},
the index $[\OG(\Lplus, \PPPplus): \vGam]$ is equal to the index 
of the image of the natural homomorphism 
$\etaminus\colon \OG(\Lminus) \to \OG(\qminus)$
in $\OG(\qminus)$.
The latter  is
computed by the method in Section~\ref{subsec:computational}, and is 
indicated in Table~\ref{table:Lpmqpm}.
This completes the proof.
\end{proof}
%
Next we investigate the walls of $D_0$.
Recall that, 
by 
the diagram~\eqref{eq:groupfibprodCenvep} and the surjectivity of 
the natural homomorphism $\OG(\Leechplus)\cong W(E_8)\to \OG(\qplus)\cong \OG(u^{\oplus 4})$, 
the projection
$\Cen(\dotz, \vep_{16})\to \OG(\Leechminus)$ is surjective.
Let $\barSigma_{\vep} ([\tau])$ denote the stabilizer of $[\tau]\in \Coker(\pi_{16})$ in $ \OG(\Leechminus)$.
Then $\barSigma_{\vep} ([\tau])$ is the image of $\Sigma_{\vep} ([\tau])\subset \Cen(\dotz, \vep_{16})$
by the projection, and 
the action of $\Sigma_{\vep} ([\tau])$
on $\tilXitau$ 
factors through the surjection $\Sigma_{\vep} ([\tau])\to\barSigma_{\vep} ([\tau])$.
Note that, by Remark~\ref{rem:K16}, the group  $K_{16}$ of order $512$ is contained in $ \barSigma_{\vep} ([\tau])$.
%
\subsubsection{Case $\tau=0$.}
In this case,
we have $\Xitau=\tilXitau$,
and it is a singleton $\{0\}$.
We obviously have $\MMM(D_0)=\tilwalls(D_0)$,
and the group $\barSigma_{\vep} ([\tau])$ acts transitively $\tilXitau$.
Hence $\Stab_{\Gamma}(D_0)$ acts  transitively on  $\walls(D_0)$.
The fiber of $\mu\colon \tilwalls(D_0)\to\walls(D_0)$ 
over a wall $\mu(\lambda)$ of $D_0$ is a singleton $\{\lambda\}$.
%
%
\subsubsection{Case $\tau=\tauI$.}
In this case,
we have $|\Xitau|=|\tilXitau|=512$,
and all vectors of $\tilXitau$ are of norm $3/2$.
We have $\MMM(D_0)=\tilwalls(D_0)$.
We confirm by direct computation that 
the  subgroup $K_{16}$  of $\barSigma_{\vep} ([\tau])$
acts on $\tilXitau$ transitively.
Hence, a fortiori, so does $\barSigma_{\vep} ([\tau])$, and 
therefore  $\Stab_{\Gamma}(D_0)$ acts  transitively on  $\walls(D_0)$.
The fiber 
 of $\mu$ over $\lambda\in \walls(D_0)$ consists 
of two vectors
$\lambda=\lambdaplus+\lambdaminus$ and $\lambda\sprime=\lambdaplus+\tau-\lambdaminus$, and 
the corresponding pair of Leech roots 
 satisfies
$\intfL{\LeechRoot{\lambda}, \LeechRoot{\lambda\sprime}}=1$.
\subsubsection{Case $\tau=\tauII$.}
In this case,
we have $|\Xitau|=|\tilXitau|=32$,
and all vectors of $\tilXitau$ are of norm $1$.
We have $\MMM(D_0)=\tilwalls(D_0)$.
The group $K_{16}$  
acts on $\tilXitau$ transitively,
and hence so does $\barSigma_{\vep} ([\tau])$.
Hence the action of $\Stab_{\Gamma}(D_0)$ on  $\walls(D_0)$ is transitive.
The fiber 
 $\mu\inv(\mu(\lambda))$ 
consists of two vectors
$\lambda=\lambdaplus+\lambdaminus$ and $\lambda\sprime=\lambdaplus+\tau-\lambdaminus$, and 
the corresponding pair of Leech roots 
 satisfies
$\intfL{\LeechRoot{\lambda}, \LeechRoot{\lambda\sprime}}=0$.
%
 %
 
 \begin{remark}\label{rem:further}
 Let us pursue the heuristic analogy given in Section~\ref{subsec:alggeom} further.
We consider $\gamma=(\vep_{16}, \tauII)$ as an Enriques involution of 
the virtual K3 surface $\XX$,
and let $\YY$ be the corresponding virtual Enriques surface.
We consider $\ConChamz$ as the nef-and-big cone of $\XX$.
Then $D_0=\PPPplus\cap \ConChamz$ is the nef-and-big cone of $\YY$.
Each wall of $D_0$ corresponds to a smooth rational curve on $\YY$,
and its pullback  splits into a disjoint pair of smooth rational curves on $\XX$.
Moreover, if we ignore the period condition, 
we can regard $\dotinf$ as the automorphism group of $\XX$ 
and $\Stab_{\vGam}(D_0)$  as the automorphism group of $\YY$.
Unfortunately, the kernel of 
the natural homomorphism $\Cen(\dotinf,\gamma)\to\Stab_{\vGam}(D_0)$
is not equal to $\angs{\gamma}$  but  of size 
\[
32=|\Ker(\OG(\Lminus) \to \OG(\qminus)|
\]
by the diagram~\eqref{eq:groupfibprodP}.
Therefore,  it seems that, for  the analogy to be perfect, 
 the period condition  should be appropriately considered.
 \end{remark}
\section*{Acknowledgements}
The author would like to thank Professor Simon Brandhorst for many valuable discussions and insightful comments, which greatly improved this paper.
This work was supported by JSPS KAKENHI Grant Numbers 20H00112, 23K20209, and 23H00081, and by the Deutsche Forschungsgemeinschaft (DFG, German Research Foundation) – SFB-TRR 195 – Project-ID 286237555.
%
 %
%
%

\bibliographystyle{plain}
\bibliography{myrefsaffineConway}
\end{document}